\documentclass[12pt]{amsart}
\usepackage{amsmath,amsthm,amsfonts,amssymb}
\usepackage{amscd}
\usepackage{colortbl}

\usepackage{hhline}
\usepackage{tabularx}

\numberwithin{equation}{section}

\theoremstyle{plain}
\newtheorem{theorem}{Theorem}

\newtheorem{lemma}[theorem]{Lemma}
\newtheorem{corollary}[theorem]{Corollary}

\theoremstyle{definition}

\theoremstyle{remark}
\newtheorem{remark}[theorem]{Remark}

\usepackage{pgf}
\usepackage{tikz}
\tikzstyle{vertex}=[circle, draw, fill=black, inner sep=0pt, minimum width=6pt]
\tikzstyle{fat}=[circle, draw=black, fill=black, inner sep=0pt, minimum width=10pt]
\tikzstyle{pedge}=[draw,-]
\tikzstyle{nedge}=[draw,densely dashed]
\tikzstyle{empty}=[circle, draw=white, inner sep=2pt, fill=white, minimum width=4pt]

\DeclareMathOperator{\supp}{supp}

\newcommand{\norm}[1]{\lVert#1\rVert}

\newcommand{\kernel}{\operatorname{ker}}

\newcommand{\R}{\mathbb{R}}

\newcommand{\Z}{\mathbb{Z}}

\usepackage{multirow}

\newcommand{\LG}{\operatorname{\mathfrak{L}}}
\newcommand{\cS}{\mathcal{S}}

\title[graphs with smallest eigenvalue greater than $-2$]
{Edge-signed graphs with smallest eigenvalue greater than $-2$}

\author[G. Greaves]{ Gary Greaves$^1$ }
\thanks{$^1$GG was supported by JSPS KAKENHI; 
grant number: 24$\cdot$02789}
\author[J. Koolen]{ Jack Koolen$^2$ }
\thanks{$^2$Part of this work was done while JHK was 
visiting Tohoku University with a JSPS visitors grant. 
He also greatly appreciates the support of the `100 talents' program 
of the Chinese Academy of Sciences. }
\author[A. Munemasa]{ Akihiro Munemasa }
\author[Y. Sano]{ Yoshio Sano$^3$ }
\thanks{$^3$YS was supported by JSPS KAKENHI; grant number: 25887007.}
\author[T. Taniguchi]{ Tetsuji Taniguchi$^{4}$ }
\thanks{$^{4}$TT was supported by JSPS KAKENHI; grant number: 25400217}

\address{Research Center for Pure and Applied Mathematics,
Graduate School of Information Sciences, 
Tohoku University, Sendai 980-8579, Japan}
\email{grwgrvs@ims.is.tohoku.ac.jp}

\address{School of Mathematical Sciences, 
USTC, and Wen-Tsun Wu Key Laboratory, CAS,
Hefei, Anhui, 230026, P.R. China}
\email{koolen@ustc.edu.cn}

\address{Research Center for Pure and Applied Mathematics, 
Graduate School of Information Sciences, 
Tohoku University, Sendai 980-8579, Japan}
\email{munemasa@math.is.tohoku.ac.jp}

\address{Division of Information Engineering, 
Faculty of Engineering, Information and Systems,  
University of Tsukuba, Ibaraki 305-8573, Japan} 
\email{sano@cs.tsukuba.ac.jp}

\address{Matsue College of Technology, Matsue 690-8518, Japan} 
\email{tetsuzit@matsue-ct.ac.jp}

\dedicatory{Dedicated to Alan J. Hoffman on the occasion of his ninetieth birthday.}

\begin{document}
\begin{abstract}
We give 
a structural classification of edge-signed graphs 
with smallest eigenvalue greater than $-2$.
We prove a conjecture of Hoffman about the smallest eigenvalue of 
the line graph of a tree 
that was stated in the 1970s.
Furthermore, we prove a more general result 
extending Hoffman's original statement 
to all edge-signed graphs with smallest eigenvalue greater than $-2$. 
Our results give a classification of the special graphs of fat Hoffman graphs
with smallest eigenvalue greater than $-3$.
\end{abstract}
\maketitle

\section{Introduction}

The (adjacency) eigenvalues of a graph $G$ on $m$ vertices are defined as the eigenvalues of its adjacency matrix $A$.
Since $A$ is a real symmetric matrix, its eigenvalues $\lambda_i(A)$ are real; we arrange them as follows
\[
	\lambda_1(A) \leqslant \lambda_2(A) \leqslant \cdots \leqslant \lambda_m(A).
\]
For convenience we will sometimes also refer to each $\lambda_i(A)$ as $\lambda_i(G)$.
Much attention has been directed towards the study of graphs with smallest eigenvalue at least $-2$ \cite{BusseNeum,Doob:LineGraphsEigen73,DoobCvetkovic:LAA79,McKee:IntSymCyc07,ArsComb1993}.
Most of this attention has centred around the beautiful theorem of Cameron, Goethals, Shult, and Seidel \cite{CGSS}, which classifies 
graphs having smallest eigenvalue at least $-2$.
In the late 1970s Hoffman \cite{Hoffman:lt-2} studied graphs $G$ with $\lambda_1(G) \geqslant -1 - \sqrt{2}$ and later Woo and Neumaier \cite{Woo:ltHoff} furthered Hoffman's work, introducing the so-called Hoffman graph.
Recently, Jang, Koolen, Munemasa, and Taniguchi \cite{-3} proposed a programme to classify fat Hoffman graphs with smallest eigenvalue at least $-3$.
The present work fills a part of this programme, and 
includes the results of \cite{gt-3}.

In this article we classify, up to switching equivalence, edge-signed graphs with smallest eigenvalue greater than $-2$. 
(See Section~\ref{sec:pre} for the definition of the eigenvalues of an edge-signed graph.)
In particular, we recover as a special case the classification of graphs with smallest
eigenvalue greater than $-2$ given earlier by Doob and Cvetkovi\'c 
\cite{DoobCvetkovic:LAA79}.
As an application, we classify the special graphs of fat Hoffman
graphs with smallest eigenvalue greater than $-3$.
Some of such graphs are related to the modified adjacency
matrix that appeared in a paper of Hoffman \cite{Hoffman:LimitLeastEig1977}.
Below we describe the conjecture Hoffman made about these modified adjacency matrices.

Let $T$ be a tree on $m \geqslant 2$ vertices with line graph $\LG(T)$ 
and let $e$ be an end-edge of $T$ (one of whose vertices has valency $1$).
Then $e$ is a vertex of $\LG(T)$.
For a graph $G$ and a vertex $v \in V(G)$, define $\hat{A}(G,v)$ to be the adjacency matrix of $G$, modified 
by putting a $-1$ in the diagonal position corresponding to $v$.
In one of his papers \cite{Hoffman:LimitLeastEig1977}, 
Hoffman conjectured that
$\lambda_{1}(\hat{A}(\LG(T),e)) < \lambda_1(\LG(T))$
for all trees $T$ and end-edges $e$.
In Section~\ref{sec:Hoffcon} we prove this conjecture, which we record as the following theorem.

\begin{theorem}\label{thm:Hoffman}
	Let $T$ be a tree and let $e$ be an end-edge of $T$.
	Then $\lambda_1(\hat{A}(\LG(T),e)) < \lambda_1(\LG(T))$.
\end{theorem}

Furthermore, using the classification of edge-signed graphs (see Theorem~\ref{thm:3}) with smallest eigenvalue greater than $-2$, we prove a generalised version of Hoffman's conjecture (see Theorem~\ref{thm:11}).

In Section~\ref{sec:pre} we give our preliminaries.
In Section~\ref{sec:class} we prove the main part of the classification theorem of edge-signed graphs with smallest eigenvalue greater than $-2$ leaving the exceptional case to Section~\ref{sec:ex.grph}.
In Section~\ref{sec:Hoffcon} we prove Theorem~\ref{thm:Hoffman} and 
its generalised version, and in Section~\ref{sec:Hoffgraph} 
we comment on the application to Hoffman graphs with smallest eigenvalue
greater than $-3$.

\section{Edge-signed graphs and representations}
\label{sec:pre}

In this section we introduce some terminology that we use in our results.
An \textbf{edge-signed graph} $G$ 
is a triple 
$(V, E^{+}, E^{-})$
of a set $V$ of vertices, 
a set $E^{+}$ of $2$-subsets  
of $V$ (called $(+)$-\textbf{edges}), and 
a set $E^{-}$ of $2$-subsets  
of $V$ (called $(-)$-\textbf{edges}) 
such that  
$E^{+} \cap E^{-} = \emptyset$. 

Let $G$ be an edge-signed graph. 
We denote 
the set of vertices of $G$ 
by $V(G)$, 
the set of $(+)$-edges of $G$ 
by $E^{+}(G)$, 
and 
the set of $(-)$-edges of $G$ 
by $E^{-}(G)$. 
By a \textbf{subgraph}
$G' = (V(G'), E^{+}(G'), E^{-}(G'))$ of $G$ we mean a vertex
induced edge-signed subgraph, i.e., 
$V(G') \subseteq V(G)$ and
$E^{\pm}(G')=\{\{x,y\} \in E^{\pm}(G) \mid x,y \in V(G')\}$.
The \textbf{underlying graph} $U(G)$ of $G$ 
is the unsigned graph 
$(V(G),E^+(G) \cup E^-(G))$.

Two edge-signed graphs $G$ and $G'$ 
are said to be \textbf{isomorphic} 
if there exists a bijection $\phi: V(G) \to V(G')$ 
such that $\{u,v\} \in E^+(G)$ if and only if 
$\{\phi(u), \phi(v) \} \in E^+(G')$ and 
that $\{u,v\} \in E^-(G)$ if and only if 
$\{\phi(u), \phi(v) \} \in E^-(G')$. 

For an edge-signed graph $G$, 
we define its \textbf{signed adjacency matrix} $A(G)$ 
by 
\[
(A(G))_{uv}=
\begin{cases}
1  & \text{if } \{u,v\} \in E^+(G), \\
-1 & \text{if } \{u,v\} \in E^-(G), \\
0  & \text{otherwise.} 
\end{cases} 
\]
To ease language, we will refer to the signed adjacency matrix simply as the adjacency matrix.
The eigenvalues of $G$ are defined to be those of $A(G)$.

A \textbf{switching} at vertex $v$ is the process of swapping the signs of each edge incident to $v$.
Two edge-signed graphs $G$ and $G'$ 
are said to be \textbf{switching equivalent} if there exists a subset $W \subset V(G)$ such that $G^\prime$ is isomorphic to the graph obtained by switching at each vertex in $W$.
Note that switching equivalence is an equivalence relation that preserves eigenvalues.

Let $G$ be an edge-signed graph with smallest eigenvalue at least $-2$.
A \textbf{representation} of $G$ is a mapping $\phi$ from $V(G)$ to $\R^n$ 
for some positive integer $n$, 
such that $(\phi(u),\phi(v)) = \pm 1$ if
$\{u,v\} \in E^\pm (G)$ respectively, and $(\phi(u),\phi(v))=2\delta_{u,v}$
otherwise, 
where 
$\delta_{u,v}$ is Kronecker's delta, i.e., 
$\delta_{u,v}=1$ if $u=v$ and $\delta_{u,v}=0$ if $u \neq v$.
Since, 
for $A$ the adjacency matrix of $G$, the matrix $A + 2I$ is positive semidefinite,
$A + 2I$ is the Gram matrix of a set $S$ of vectors $\mathbf{x}_1, \dots, \mathbf{x}_m$.
These vectors satisfy $(\mathbf{x}_i, \mathbf{x}_i) = 2$ and 
$(\mathbf{x}_i, \mathbf{x}_j) = 0,\pm 1$ for $i \ne j$.
Sets of vectors satisfying these conditions 
determine 
line systems.
We denote by $[\mathbf{x}]$ the line determined by a nonzero vector
$\mathbf{x}$, in other words, $[\mathbf{x}]$ is the one-dimensional
subspace spanned by $\mathbf{x}$.
We say that $G$ is \textbf{represented by} the line system $S$ if $G$ has a representation $\phi$ such that $S = \{ [\phi(v)] : v \in V(G) \}$.

Below we give descriptions of three line systems, $A_n$, $D_n$ and $E_8$.
Let $\mathbf{e}_1, \dots, \mathbf{e}_n$ be an orthonormal basis for $\R^n$. 
\begin{align*}
	A_n &=
	\left \{ [ \mathbf{e}_i - \mathbf{e}_j ] 
	: 1 \leqslant i < j \leqslant n+1 \right \}\quad(n\geqslant1), \\
	D_n &= A_{n-1} \cup \left \{ [ \mathbf{e}_i + \mathbf{e}_j ] 
	: 1 \leqslant i < j \leqslant n \right \}\quad(n\geqslant4),\\
	E_8 &= D_8 \cup 
	\left \{ [ \frac{1}{2} \sum_{i=1}^8 \epsilon_i\mathbf{e}_i ] 
	: \epsilon_i = \pm 1, \prod_{i=1}^8 \epsilon_i = 1 \right \}.
\end{align*}
These line systems are used in the following classical result of Cameron, Goethals, Shult, and Seidel.

\begin{theorem}[\cite{CGSS}]
	\label{thm:CGSS}
	Let $G$ be a connected edge-signed graph with $\lambda_1(G) \geqslant -2$.
	Then $G$ is represented by a subset of either $D_n$ or $E_8$.
\end{theorem}

Let $G$ be an edge-signed graph represented by a line system $S$.
If $S$ embeds into $\Z^n$ for some $n$, then we say that $G$ is 
\textbf{integrally represented} or that $G$ has an 
\textbf{integral representation}.
By Theorem~\ref{thm:CGSS}, for an edge-signed graph $G$ with
$\lambda_1(G) \geqslant -2$, $G$ has an integral representation if and only if 
$G$ is represented by a subset of $D_n$ for some $n$,
or equivalently, $G$ is represented by $D_\infty$, in
the sense of \cite{EJC1987}. 
We record this observation as the following corollary.

\begin{corollary}\label{cor:integralReps}
	Let $G$ be a connected edge-signed graph with $\lambda_1(G) \geqslant -2$.
	Then $G$ has no integral representation if and only if 
	$G$ is represented by a subset of $E_8$
	but not by a subset of $D_n$ for any $n$.
\end{corollary}

Corollary~\ref{cor:integralReps} motivates our next definition. 
Let $G$ be a connected edge-signed graph with $\lambda_1(G) \geqslant -2$.
We call $G$ \textbf{exceptional} if it does not have an integral representation.
Clearly there are only finitely many exceptional edge-signed graphs.

\section{Classification of edge-signed graphs with $\lambda_1 > -2$}
\label{sec:class}

In this section we classify integrally represented edge-signed graphs with smallest eigenvalue greater than $-2$.
We leave the exceptional case until Section~\ref{sec:ex.grph}.

\begin{lemma}\label{lem:signed-cycle}
Let $G$ be an edge-signed graph whose underlying graph is a cycle.
Then $\lambda_{1}(G)>-2$ if and only if the number of $(+)$-edges of
$G$ is odd.
\end{lemma}
\begin{proof}
If the number of $(+)$-edge is even, then $G$ is switching equivalent
to the edge-signed cycle in which all edges are $(-)$-edges,
hence $\lambda_{1}(G)=-2$. Conversely, suppose $G$ has an odd
number of $(+)$-edges. If the length of $G$ is odd, then $G$ has an
even number of $(-)$-edges, hence $G$ is switching equivalent to
an unsigned cycle. Thus $\lambda_{1}(G)>-2$. 
If the length of $G$ is even, then, up to switching, $A(G) = B^\top B - 2I$ where 
\[
B = \begin{pmatrix}
1&0&\cdots&0&-1\\
1&1&\ddots&&0\\
0&1&1&\ddots&\vdots\\
\vdots&&\ddots&\ddots&0\\
0&\cdots&\cdots&1&1
\end{pmatrix}.
\]
Since this matrix is nonsingular, $\lambda_{1}(G)>-2$. 
\end{proof}

Let $G$ be a graph with an orientation assigned to its edges.
We define the \textbf{oriented incidence matrix} $B = B(G)$ of $G$ to be the $\{0, \pm 1 \}$-matrix whose rows and columns are indexed by $V(G)$ and $E(G)$ respectively such that the entry $B_{ve}$ is equal to $1$ if $v$ is the head of $e$, $-1$ if $v$ is the tail of $e$, and $0$ otherwise.
See \cite{Godsil:AlgGraph} for properties of oriented incidence matrices.

Let $G$ be an edge-signed graph with smallest eigenvalue 
greater than 
$-2$. Assume that $G$ has an integral representation $\phi$.
This means that, with $m=|V(G)|$,
there exists an $n\times m$ matrix 
\[
M=\begin{pmatrix} \mathbf{v}_1&\cdots&\mathbf{v}_m\end{pmatrix},
\]
with entries in $\Z$, such that 
$(\mathbf{v}_i,\mathbf{v}_j) = \pm 1$ if $\{i,j\} \in E^\pm (G)$ respectively, and
$(\mathbf{v}_i,\mathbf{v}_j)= 2\delta_{i,j}$ otherwise. 
We may assume that $M$ has no rows consisting only of zeros.
Since $\mathbf{v}_i\in\Z^n$, $\mathbf{v}_i$ has two entries equal to
$\pm1$, and all other entries $0$. This means that we can regard
$M$ as an oriented incidence matrix of a graph $H$ and the underlying
graph of $G$ is the line graph of $H$. 
More precisely, $H$ is a graph with $n$ vertices, and the vertices $i$ and $j$ are joined by the edge $k$ whenever $\{i,j\}=\supp(\mathbf{v}_k)$.
Note that the graph $H$ may have multiple edges.
A graph without multiple edges is called \textbf{simple}.
We call $H$ the \textbf{representation graph} of $G$ associated with the
representation $\phi$.
Note that $H$ has no isolated vertex. If $G$ is connected, then so is $H$.

\begin{lemma}\label{lem:1}
Let $G$ be an $m$-vertex connected edge-signed graph having an integral representation $\phi$ and smallest eigenvalue greater than $-2$.
Let $H$ be the $n$-vertex representation graph of $G$ associated with the
representation $\phi$.
Then $n\in\{m,m+1\}$. 
Moreover, if $n=m$, then $H$ is a unicyclic graph or a tree with a double edge and if $n=m+1$, then $H$ is a tree.
\end{lemma}
\begin{proof}
Since $M^\top M=A(G)+2I$ is positive definite, $M$ has rank $m$. This implies
$n\geqslant m$. If $H$ is disconnected, then so is $G$, which is absurd. Thus
$H$ is connected, which forces $n\leqslant m+1$.
\end{proof}

Let $H$ be a unicyclic graph whose unique cycle $C$ has at least $4$ vertices and let $G = \LG(H)$.
Then for each edge $e$ of $G$ there exists a unique 
maximal clique 
that contains $e$.
For such a graph $G$, we denote by $\mathfrak C_G(e)$ the unique maximal clique of $G$ containing the edge $e$. 
Let $uu'$ be an edge of $\LG(C)$.
Define $\LG^\dag(H,uu^\prime)$ to be the edge-signed graph $(V,E^{+},E^-)$, where $V = V(\LG(H))$, 
\[
E^- = \left \{ uv \in E(\LG(H)) \mid v \in \mathfrak C_G(uu^\prime) \right \}
\]
and $E^+ = E(\LG(H)) \backslash E^-$.
Observe that, for all edges $uu^\prime$ and $vv^\prime$ of $\LG(C)$, the graph $\LG^\dag(H,uu^\prime)$ is switching equivalent to $\LG^\dag(H,vv^\prime)$.
%

\begin{theorem}\label{thm:3}
Let $G$ be a connected integrally represented edge-signed graph having smallest eigenvalue greater than $-2$.
Let $H$ be the representation graph of $G$ for some integral representation.
Then one of the following statements holds:
\begin{enumerate}
\item
$H$ is a simple tree or $H$ is unicyclic with an odd cycle, and
$G$ is switching equivalent to the line graph $\LG(H)$, 
\item
$H$ is unicyclic with an even cycle $C$,
and $G$ is switching equivalent to 
$\LG^\dag(H,uu^\prime)$
where $uu'$ is an edge of $\LG(C)$.
\item 
$H$ is a tree with a double edge, and 
$G$ is switching equivalent to 
the edge-signed graph obtained from the line graph $\LG(H)$, 
by attaching a new vertex $u'$, and join $u'$ by $(+)$-edges to 
every vertex of a clique in the neighbourhood of $u$,
$(-)$-edges to every vertex of the other clique in the neighbourhood of $u$,
where $u$ is a fixed vertex of $\LG(H)$.
\end{enumerate}
Conversely, if $G$ is an edge-signed graph described by
{\rm(i)}, {\rm(ii)}, or {\rm(iii)} above, then $G$ is integrally represented and has smallest
eigenvalue greater than $-2$.
\end{theorem}
\begin{proof}
	By Lemma~\ref{lem:1}, we can divide the proof into three cases.
	
	\paragraph{Case 1: $H$ is a simple tree} 
	\label{par:case_1_h_is_a_tree}
	Since $G$ has smallest eigenvalue greater than $-2$, $G$ cannot contain
	a triangle switching equivalent to one with three $(-)$-edges.
	By repeatedly applying switching, one can move
	the locations of $(-)$-edges toward an end block, and eventually 
	end up with an unsigned graph. Therefore, $G$ is switching equivalent
	to $\LG(H)$.
	
	\paragraph{Case 2: $H$ is unicyclic} 
	\label{par:case_2_h_is_unicyclic}
	We prove either (i) or (ii) holds
	by induction on the number of vertices of $H$ minus the length
	of the cycle in $H$. First suppose that $H$ is a cycle.
	By Lemma~\ref{lem:signed-cycle}, $G$ is either an odd cycle with an 
	even number of $(-)$-edges, or $G$ is an even cycle with odd
	number of $(-)$-edges. In the former case, $G$ is switching equivalent
	to an unsigned odd cycle.
	In the latter case, $G$ is
	switching equivalent to an even cycle with one $(-)$-edge and the clique of the underlying graph of $G$ containing the $(-)$-edge is then switching equivalent to the one described in (ii).

	Now suppose that $H$ is not a cycle. 
	Then $H$ has a vertex
	$v$ of degree one, adjacent to a vertex $w$.
	Let $\phi$ be the integral representation of $G$ to which
	$H$ is associated.
	Let $H'$ be the graph obtained from $H$
 	by removing the vertex $v$, and let $G^\prime$ be the graph obtained from $G$ by removing the vertex $x$ corresponding to the edge $vw$ in $H$.
	Clearly $H^\prime$ is the representation graph of $G^\prime$ associated to $\phi$ restricted to $G^\prime$ and, by induction, $H^\prime$ and $G^\prime$ satisfy either (i) or (ii).
	If $H^\prime$ and $G^\prime$ satisfy (i) then, since $H$ does not have any double edges, each nonzero entry $M_{w,y}$ (for $y \in V(G^\prime)$) has the same sign.
	On the other hand, if $H^\prime$ and $G^\prime$ satisfy (ii), that is, $G^\prime$ is switching equivalent to $\LG^\dag(H^\prime,uu^\prime)$ then, without loss of generality we can assume that $w$ is not incident to the edge $u$ of $H$.
	Now, again, since $H$ does not have any double edges, we observe that each nonzero entry $M_{w,y}$ (for $y \in V(G^\prime)$) has the same sign.
	Since $v$ has degree one, we are free to switch the vertex $x$ so that either (i) or (ii) holds for $G$ and $H$. 
	
	\paragraph{Case 3: $H$ is a tree with a double edge} 
	\label{par:case_3_h_has_a_double_edge}
		If $H$ has a double edge, then the matrix $M$ has a submatrix
		\[
		\begin{pmatrix} 1&1\\ -1&1\end{pmatrix}.
		\]
		Let $u'$ (resp. $u$) denote the vertex of $G$ corresponding to
		the first (resp. second) column of this matrix (which in turn
		corresponds to a column of $M$), and let
		$v^+$ (resp. $v^-$) denote the vertex of $H$ corresponding to
		the first (resp. second) row of this matrix (which in turn
		corresponds to a row of $M$).
		Let $G'=G-u'$. Then the graph $H'$ obtained from $G'$ is 
		a tree. We have already shown that, in this case, we may take
		$G'$ to be the unsigned line graph of $H'$. Let
		$K^+$ (resp. $K^-$) be the clique of $G$ in the neighbourhood of
		$u$ consisting of vertices $u''$ with $M_{v^+,u''}=1$
		(resp. $M_{v^-,u''}=1$). Then in the graph $G$, $u'$ is joined
		to every vertex of $K^+$ by $(+)$-edges, and $u'$ is joined
		to every vertex of $K^-$ by $(-)$-edges. Therefore (iii) holds.
	

Conversely, suppose $G$ is described by (i), (ii), or (iii).
First, we describe how to construct $M$ for each case.
For (i), $M$ is the incidence matrix of $H$.
For (ii), let $v$ be the vertex incident to both the edges $u$ and $u^\prime$ of $H$.
Then $M$ is the incidence matrix of $H$ adjusted so that $M_{v,u}=-1$.
For (iii), let $v$ and $w$ be incident to the edge $u$ in $H$.
Then $M$ is the incidence matrix of $H$ together with an extra column for $u^\prime$ with $M_{v,u^\prime}=1$, $M_{w,u^\prime}=-1$, and the remaining entries $0$.

To prove the converse, it suffices to show that, in each case, $M^\top M$ is positive definite.
If $G$ is the line graph of a tree then this is well known.
Thus we can immediately restrict our attention to when $n = m$.
We will show that, in both remaining cases, $M^\top M$ has determinant $4$.
We inductively show that $\det(M^\top M) = 4$ for $H$ unicyclic.
Suppose that the underlying graph of $G$ is the line graph of a unicyclic graph.
If $H$ is a cycle then, by Lemma~\ref{lem:signed-cycle}, $M$ is nonsingular.
Hence the rows of $M$ are a basis for $D_n$, which has discriminant $4$.
Thus $M$ has determinant $\pm 2$.
Otherwise, $H$ has a vertex $v$ of degree $1$.
Let $M^\prime$ be a the matrix obtained by removing $v$.
Then $\det(M) = \pm\det(M^\prime)$.
Hence $M^\top M$ is positive definite, as required.
The same inductive approach can be applied when starting with the double-edge where $M$ is the matrix
\[
\begin{pmatrix} 1&1\\ -1&1\end{pmatrix},
\]
which has determinant $2$.
\end{proof}

Note that, if $G$ is represented by the line system $A_n$, then
one can relax the assumption of Theorem~\ref{thm:3} to having
smallest eigenvalue \emph{at least} $-2$.
Ishihara~\cite{Ishihara} shows that, in this case, the
underlying graph of $G$ is a claw-free block graph.

\section{Hoffman's conjecture}
\label{sec:Hoffcon}

In this section we settle Hoffman's conjecture, i.e., we prove Theorem~\ref{thm:Hoffman}.
Moreover, we prove a stronger version of Hoffman's conjecture extended to edge-signed graphs.


%

\begin{lemma}\label{lem:v1zero}
Let $A=(a_{i,j})$ be a real symmetric matrix, and let $A'=(a'_{i,j})$
be the matrix defined by $a'_{i,j}=a_{i,j}-\delta_{i,1}\delta_{j,1}$.
Suppose there exists an eigenvector $\mathbf{x}$ of $A$ belonging to the
eigenvalue $\lambda_1(A)$ with $x_1\neq0$. Then
$\lambda_1(A) > \lambda_1(A^\prime)$.
\end{lemma}
\begin{proof}
We may assume without loss of generality that $\norm{\mathbf{x}} = 1$.
Then $\lambda_1(A) = 
\mathbf{x}^\top A \mathbf{x}=
\mathbf{x}^\top A^\prime \mathbf{x} + x_1^2 
\geqslant \lambda_1(A^\prime) + x_1^2>\lambda_1(A^\prime)$.
\end{proof}

\begin{remark}
One might wonder if Theorem~\ref{thm:Hoffman} can be proved by
showing that the adjacency matrix of $\LG(T)$ satisfies the
assumption of Lemma~\ref{lem:v1zero} when we take the first
entry to correspond to an end-edge of a tree $T$. This approach, however,
does not work. In fact, let $T$ be the Dynkin diagram $E_6$,
and let the first entry of $A$ correspond to the unique end
edge attached to the vertex of degree $3$. Then the smallest
eigenvalue $-(\sqrt{5}+1)/2$ of 
$\LG(T)$ has multiplicity $1$, and the 
eigenvector has $0$ in its first entry.
Hence, $E_6$ is an example of a graph to which we cannot apply Lemma~\ref{lem:v1zero}.
\end{remark}

\begin{lemma}\label{lem:LGtoL}
	Let $M$ be an $n \times m$ real matrix.
	Then $M^\top M$ and $MM^\top$ have the same nonzero eigenvalues
	(including multiplicities). More explicitly, for any nonzero
	eigenvalue $\theta$ of $M^\top M$, 
the multiplication by $M$ gives a linear map from
$\kernel(M^\top M-\theta I)$ to $\kernel(MM^\top-\theta I)$ whose
inverse is given by $\mathbf{v}\mapsto\theta^{-1}M^\top\mathbf{v}$.
\end{lemma}
\begin{proof}
	Follows from 
	\cite[Lemma 2.9.2]{Brouwer:SpectraGraphs}.
\end{proof}

It is easy to see that if $G$ is a bipartite graph then there exists an orientation of $G$ such that the oriented incidence matrix $B$ of $G$ satisfies $B^\top B = A(\LG(G)) + 2I$.

\begin{lemma}\label{lem:LT}
Let $G$ be a connected bipartite graph on $m$ vertices and $n$ edges,
with oriented incidence matrix $B$ satisfying $B^\top B=A+2I$ where $A$ is the adjacency matrix of the line graph of $G$,
and let $L$ be the Laplacian matrix of $G$.
Then for each $i\in\{2,\dots,m\}$,
$B\kernel(A-\lambda_{i+n-m}(A)I)=\kernel(L-\lambda_{i}(L) I)$
and 
$\kernel(A-\lambda_{i+n-m}(A)I)=B^\top\kernel(L-\lambda_{i}(L) I)$.
\end{lemma}
\begin{proof}
Since $G$ is connected, the multiplicity of $0$ as an eigenvalue of
$L$ is $1$. 
Since $B^\top B=A+2I$ and $BB^\top=L$,
Lemma~\ref{lem:LGtoL} 
implies that $\lambda_{i}(L)=\lambda_{i+n-m}(A+2I)
=\lambda_{i+n-m}(A)+2$ for $1<i\leq m$.
Moreover, 
Lemma~\ref{lem:LGtoL} 
implies 
$B\kernel(B^\top B-\lambda_{i+n-m}(B^\top B)I)
=\kernel(L-\lambda_{i}(L) I)$
and 
$\kernel(B^\top B-\lambda_{i+n-m}(B^\top B)I)
=B^\top\kernel(L-\lambda_{i}(L) I)$.
Since $B^\top B-\lambda_{i+n-m}(B^\top B)I=A-\lambda_{i+n-m}(A)I$, we obtain
the desired result.
\end{proof}

Let $G$ be a graph and let $v$ be a vertex of $G$. 
Recall that $\hat{A}(G,v)$ is the adjacency matrix of $G$, modified 
by putting a $-1$ in the diagonal position corresponding to $v$.

\begin{lemma}[{\cite[Lemma 2.1]{Hoffman:LimitLeastEig1977}}]\label{lem:gtminus22}
Let $T$ be a tree and let $e$ be an end-edge of $T$. Then
$\lambda_1(\hat{A}(\LG(T),e))>-2$.
\end{lemma}

We are now ready to prove Theorem~\ref{thm:Hoffman}.

\begin{proof}[Proof of Theorem~\ref{thm:Hoffman}]
Let $T$ be a tree on $n+1$ vertices and $n$ edges, 
and let $A = (a_{i,j})$ denote the adjacency matrix of the line
graph $\LG(T)$ of $T$.
Since $T$ is bipartite, 
one can orient its edges so that its oriented incidence matrix $B = (b_{i,j})$ satisfies
\[
	B^\top B = A + 2I.
\]
We also have $BB^\top = L$, the Laplacian matrix of $T$.


Let $r$ and $s$ be the vertices of the end-edge $e$, and
assume $r$ has valency $1$ in $T$.
We may choose $B$ so that the first row and column correspond to $r$ and $e$
respectively, and the second row corresponds to $s$.

Let the column vectors $\mathbf{e}_i$ (resp. $\mathbf{f}_i$) be the canonical bases of 
the Euclidean spaces $\R^{n}$ (resp. $\R^{n+1}$).
Without loss of generality, we assume $B \mathbf{e}_1 = 
\mathbf{f}_1 - \mathbf{f}_2$ and 
$B^\top \mathbf{f}_1 = \mathbf{e}_1$.


We obtain the matrix $C$ from $B$ by setting $b_{1,1} = 0$.
Define matrices $A^\prime$ and $L^\prime$ by
\[
	C^\top C = A^\prime + 2I, \text{ and } CC^\top = L^\prime.
\]
Then $A^\prime$ can be obtained from $A$ by setting $a_{1,1}=-1$, 
that is, $A^\prime=\hat{A}(\LG(T),e)$.
The matrix $L^\prime$ can be obtained from $L$ by setting all entries of the first
row and column to zero.

By Lemma~\ref{lem:gtminus22}, $C^\top C$ is positive definite.
It then follows from Lemma~\ref{lem:LGtoL} that $L'$ is 
a positive semidefinite
$(n+1)\times (n+1)$ matrix with rank $n$.
Let $X$ be the principal submatrix of $L'$ obtained by removing
the first row and column of $L'$.
Since the matrix $L'$ has only zeros in its first row and column,
the matrix $X$ is positive definite.


Moreover, $X$ is an M-matrix, that is, in addition to being
positive definite, its off-diagonal entries are non-positive.
By \cite[Theorem 2.5.3]{HornJohnson},
$X^{-1}$ is a non-negative matrix. 
By the Perron-Frobenius theorem (see, for example, \cite{Godsil:AlgGraph}), 
any eigenvector corresponding to the smallest eigenvalue of $X$ has no zero entry.

%

By Lemma~\ref{lem:LGtoL}, 
$\lambda_1(A'+2I)=\lambda_2(L')=\lambda_1(X)$,
and $\lambda_1(A+2I)=\lambda_2(L)$. 
Thus, to prove 
Theorem~\ref{thm:Hoffman},
it suffices to show that $\lambda_1(X) < \lambda_2(L)$.
	By Lemma~\ref{lem:v1zero}, we can assume 
$\kernel (A-\lambda_1(A)I) \subset \mathbf{f}_1^\perp$.
	
	Let $\mathbf{w}$ be an eigenvector of $L$ belonging to the eigenvalue
	$\lambda_2(L)$.
Then
\begin{align*}
B^\top\mathbf{w}
&\in B^\top\kernel(L-\lambda_2(L)I)
\\&=\kernel(A-\lambda_1(A)I)
&&\text{(by Lemma~\ref{lem:LT})}
\\&\subset\mathbf{e}_1^\perp.
\end{align*}

Thus
$\mathbf{w}\in (B\mathbf{e}_1)^\perp=(\mathbf{f}_1-\mathbf{f}_2)^\perp$.

Therefore $w_1 = w_2$.

	
		On the other hand, again by Lemma~\ref{lem:LT}, the eigenvector 
	$\mathbf{w}$ can be written as $B \mathbf{v}$ where $\mathbf{v}$
	is in $\kernel (A-\lambda_1(A)I) \subset \mathbf{e}_1^\perp$.
	Then it follows that $w_1=\mathbf{f}_1^\top B\mathbf{v}=\mathbf{e}_1^\top\mathbf{v}
	=0$.
	
	Hence $w_1 = w_2 = 0$.

Since the first row of $L$ can be written as $\mathbf{f}_1^\top - \mathbf{f}_2^\top$, we have $$
	L \mathbf{w} = \begin{pmatrix}
		0 \\
		X \mathbf{y}
	\end{pmatrix},$$
	where $\mathbf{w}^\top = (0, \mathbf{y}^\top)$. 

	Hence, $\mathbf{w}$ restricts to an eigenvector $\mathbf{y}$ of $X$.
	But the first entry of $\mathbf{y}$ is zero.
	Since $X$ is an M-matrix, 
$\lambda_2(L)$ is not the smallest eigenvalue of $X$. This implies
$\lambda_1(X) < \lambda_2(L)$.
\end{proof}

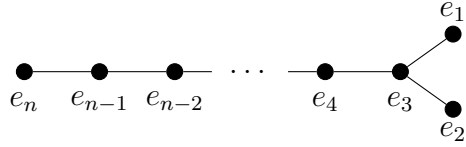
\begin{figure}[htbp]
	\centering
	\begin{tikzpicture}[yshift=4cm, xshift=0.8cm]
		\foreach \type/\pos/\name in {{vertex/(0,0.5)/bgn}, {vertex/(1,0.5)/b1}, {vertex/(2,0.5)/e1}, {empty/(2.6,0.5)/b11}, {empty/(3.4,0.5)/c11}, {vertex/(4,0.5)/c1}, {vertex/(5,0.5)/d1}, {vertex/(5.7,1)/f1}, {vertex/(5.7,0)/f2}}
			\node[\type] (\name) at \pos {};
		\foreach \pos/\name in {{(0,0.1)/e_n}, {(1,0.1)/e_{n-1}}, {(2,0.1)/e_{n-2}}, {(4,0.1)/e_4}, {(5,0.1)/e_3}, {(5.7,-0.3)/e_2}, {(5.7,1.3)/e_1}}
			\node at \pos {$\name$};
		\foreach \pos/\name in {{(3,0.5)/\dots}}
			\node at \pos {$\name$};
		\foreach \edgetype/\source/ \dest in {pedge/b1/e1, pedge/e1/b11, pedge/c11/c1, pedge/c1/d1, pedge/d1/f1, pedge/d1/f2}
			\path[\edgetype] (\source) -- (\dest);
		\foreach \edgetype/\source/\dest in {pedge/b1/bgn}
			\path[\edgetype] (\source) -- (\dest);
	\end{tikzpicture}
	\caption{The graph $\mathcal X_n^{(1)} (n \geqslant 3)$}
	\label{fig:x1}
\end{figure}

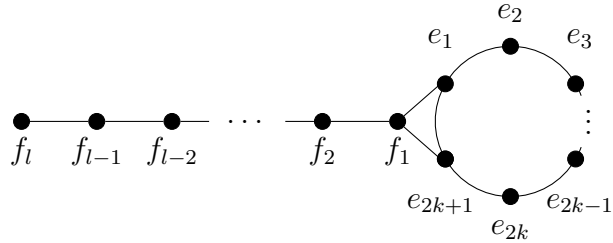
\begin{figure}[htbp]
	\centering
	\begin{tikzpicture}
		\begin{scope}[xshift=6.5cm, yshift=0.5cm]
			\newdimen\rad
			\rad=1cm
			\newdimen\radi
			\radi=1.04cm
			\newdimen\radii
			\radii=1.44cm
			\draw (365.5:\radi) node[empty] {$\vdots$};
			\foreach \y/\lab in {70/e_{2k-1},170/e_3,210/e_2,250/e_1,350/e_{2k+1},390/e_{2k}}
			{
				\def\x{\y - 120}
				\draw (\x:\radii) node[empty] {$\lab$};
		    }
			\foreach \y in {90,150,210,270,330,390}
			{
				\def\x{\y - 120}
				\draw (\x:\rad) node[vertex] {};
		    }
				\foreach \y in {210,270,330,390}
				{
					\def\x{\y - 120}
					\draw[pedge] (\x:\rad) arc (\x:\x+60:\rad);
			    }
			\draw[pedge] (30:\rad) arc (30:90:\rad);
			\draw (210:\rad) node[vertex] {};
			\draw[pedge] (330:\rad) arc (330:340:\rad);
			\draw[pedge] (380:\rad) arc (380:390:\rad);
			\node (x) at (140:\rad) {};
			\node (y) at (220:\rad) {};
		\end{scope}
		\begin{scope}
			\foreach \type/\pos/\name in {{vertex/(0,0.5)/bgn}, {vertex/(1,0.5)/b1}, {vertex/(2,0.5)/e1}, {empty/(2.6,0.5)/b11}, {empty/(3.4,0.5)/c11}, {vertex/(4,0.5)/c1}, {vertex/(5,0.5)/d1}}
				\node[\type] (\name) at \pos {};
			\foreach \pos/\name in {{(0,0.1)/f_l}, {(1,0.1)/f_{l-1}}, {(2,0.1)/f_{l-2}}, {(4,0.1)/f_2}, {(5,0.1)/f_1}}
				\node at \pos {$\name$};
			\foreach \pos/\name in {{(3,0.5)/\dots}}
				\node at \pos {$\name$};
			\foreach \edgetype/\source/ \dest in {pedge/b1/e1, pedge/e1/b11, pedge/c11/c1, pedge/c1/d1, pedge/d1/x, pedge/d1/y}
				\path[\edgetype] (\source) -- (\dest);
			\foreach \edgetype/\source/\dest in {pedge/b1/bgn}
				\path[\edgetype] (\source) -- (\dest);
		\end{scope}
	\end{tikzpicture}
	\caption{The graph $\mathcal X_{k,l}^{(2)} (k,l \geqslant 1)$}
	\label{fig:x2}
\end{figure}

\begin{figure}[htbp]
	\centering
	\begin{tikzpicture}[yshift=-5cm]
		\begin{scope}[xshift=6.5cm, yshift=0.5cm]
			\newdimen\rad
			\rad=1cm
			\newdimen\radi
			\radi=1.04cm
			\newdimen\radii
			\radii=1.44cm
			\draw (365.5:\radi) node[empty] {$\vdots$};
			\foreach \y/\lab in {70/e_{2k},170/e_3,210/e_2,250/e_1,350/e_{2k+2},390/e_{2k+1}}
			{
				\def\x{\y - 120}
				\draw (\x:\radii) node[empty] {$\lab$};
		    }
			\foreach \y in {90,150,210,270,330,390}
			{
				\def\x{\y - 120}
				\draw (\x:\rad) node[vertex] {};
		    }
				\foreach \y in {150,210,330,390}
				{
					\def\x{\y - 120}
					\draw[pedge] (\x:\rad) arc (\x:\x+60:\rad);
			    }
			\draw[nedge] (150:\rad) arc (150:210:\rad);
			\draw (210:\rad) node[vertex] {};
			\draw[pedge] (330:\rad) arc (330:340:\rad);
			\draw[pedge] (380:\rad) arc (380:390:\rad);
			\node (x) at (140:\rad) {};
			\node (y) at (220:\rad) {};
		\end{scope}
		\begin{scope}
			\foreach \type/\pos/\name in {{vertex/(0,0.5)/bgn}, {vertex/(1,0.5)/b1}, {vertex/(2,0.5)/e1}, {empty/(2.6,0.5)/b11}, {empty/(3.4,0.5)/c11}, {vertex/(4,0.5)/c1}, {vertex/(5,0.5)/d1}}
				\node[\type] (\name) at \pos {};
			\foreach \pos/\name in {{(0,0.1)/f_l}, {(1,0.1)/f_{l-1}}, {(2,0.1)/f_{l-2}}, {(4,0.1)/f_2}, {(5,0.1)/f_1}}
				\node at \pos {$\name$};
			\foreach \pos/\name in {{(3,0.5)/\dots}}
				\node at \pos {$\name$};
			\foreach \edgetype/\source/ \dest in {pedge/b1/e1, pedge/e1/b11, pedge/c11/c1, pedge/c1/d1, nedge/d1/x, pedge/d1/y}
				\path[\edgetype] (\source) -- (\dest);
			\foreach \edgetype/\source/\dest in {pedge/b1/bgn}
				\path[\edgetype] (\source) -- (\dest);
		\end{scope}
	\end{tikzpicture}
	\caption{The graph $\mathcal X_{k,l}^{(3)} (k,l \geqslant 1)$}
	\label{fig:x3}
\end{figure}
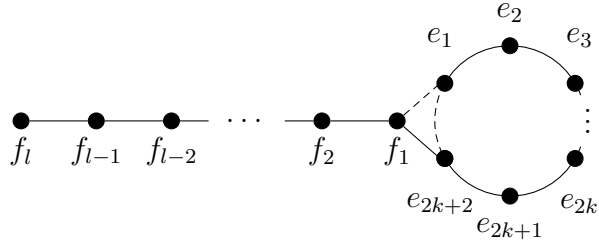

\begin{lemma}\label{lem:-2eig}
The matrices 
$\hat{A}(\mathcal{X}_n^{(1)},e_n)$,
$\hat{A}(\mathcal{X}_{k,l}^{(2)},f_l)$, and
$\hat{A}(\mathcal{X}_{k,l}^{(3)},f_l)$ (see Figures~\ref{fig:x1}, \ref{fig:x2}, and \ref{fig:x3}) have smallest eigenvalue $-2$.
\end{lemma}
\begin{proof}
	We give the eigenvectors corresponding to the eigenvalue $-2$ of 
each matrix in the statement of the lemma.
	\begin{itemize}
		\item $\hat{A}(\mathcal{X}_n^{(1)},e_n)$: 
		set $e_1, e_2 = -1$, and $e_j = (-1)^{j+1}\cdot 2$ for $j \in \{3,\dots, n\}$.
		\item $\hat{A}(\mathcal{X}_{k,l}^{(2)},f_l)$: 
		set $e_j = (-1)^{j+1}$ for $j \in \{1,\dots, 2k+1\}$ and set $f_j = (-1)^j\cdot 2$ for $j \in \{1,\dots, l\}$.
		\item $\hat{A}(\mathcal{X}_{k,l}^{(3)},f_l)$: 
		set $e_j = (-1)^{j}$ for $j \in \{1,\dots, 2k+2\}$ and set $f_j = (-1)^j\cdot 2$ for $j \in \{1,\dots, l\}$.
	\end{itemize}
Deleting the row and column containing the $-1$ on the diagonal,
we obtain the adjacency matrix of a graph with smallest eigenvalue
greater than $-2$. This is immediate for $\mathcal X_n^{(1)}$ since
the obtained graph has spectral radius less than $2$. As for
$\mathcal{X}_{k,l}^{(2)}$ and $\mathcal{X}_{k,l}^{(3)}$, the result
follows from (i) and (ii), respectively, of Theorem~\ref{thm:3}.
By interlacing, 
$\hat{A}(\mathcal{X}_n^{(1)},e_n)$,
$\hat{A}(\mathcal{X}_{k,l}^{(2)},f_l)$, and
$\hat{A}(\mathcal{X}_{k,l}^{(3)},f_l)$ 
have at most one eigenvalue less than or equal to $-2$.
This implies that $-2$ is indeed the smallest eigenvalue.
\end{proof}

\begin{theorem}\label{thm:E8}
	Let $G$ be a connected edge-signed graph and let $v\in V(G)$ such that $\lambda_{1}(\hat{A}(G,v)) \geqslant -2$.
	Then $G$ is integrally represented.
\end{theorem}
\begin{proof}
	Suppose that $\hat{A}(G,v) + 2I$ is positive semidefinite.
	Then we can write $\hat{A}(G,v) + 2I = U^\top U$ for some matrix $U$.
	Label the columns of $U$ as $\mathbf{u}_1, \dots, \mathbf{u}_n$ where $\norm{\mathbf{u}_1} = 1$ and $\norm{\mathbf{u}_i}^2 = 2$ for $i \in \{2,\dots, n\}$.
	Let $\Lambda = \bigoplus_{i=1}^{n}\Z \mathbf{u}_i$ and let $B = \{ \mathbf{v} \in \Lambda \mid \norm{\mathbf{v}} =1 \}$.
	Clearly $B = \{\pm \mathbf{v}_1, \dots, \pm \mathbf{v}_m\}$ for some $m$ with $(\mathbf{v}_i,\mathbf{v}_j) = \delta_{ij}$.
	Define $\Lambda^\prime$ as the $\Z$-span of the vectors of $B$ and set $X = \Lambda \cap (\Lambda^\prime)^\perp$.
	It is easily checked that a vector $\mathbf{v} \in \Lambda$ with $\norm{\mathbf{v}}^2 = 2$ and $\mathbf{v} \not \in \Lambda^\prime$ is orthogonal to $\Lambda^\prime$.
	Hence we can write $\Lambda$ as the orthogonal sum $\Lambda = \Lambda^\prime \perp X$ and so $\mathbf{v} \in X$.
	Unless either $\Lambda^\prime = 0$ or $X = 0$, this orthogonal decomposition of $\Lambda$ violates our assumption that $G$ is connected.
	Since $\mathbf{u}_1 \in \Lambda^\prime$,
	we must have $X = 0$.
	Therefore $\Lambda = \Lambda^\prime \cong \Z^m$.
	
	Finally, the vectors
	\[
		\begin{pmatrix}
			1 \\ \mathbf{u}_1, 
		\end{pmatrix},
		\begin{pmatrix}
			0 \\ \mathbf{u}_2, 
		\end{pmatrix},
		\dots,
		\begin{pmatrix}
			0 \\ \mathbf{u}_n 
		\end{pmatrix}
	\]
	all have norm $\sqrt{2}$ and their Gram matrix gives $A + 2I$.
	Clearly these vectors are contained in a $\Z$-lattice and this lattice represents $G$.
\end{proof}

\begin{remark}
The proof of Theorem~\ref{thm:E8} is essentially the same as
that of \cite[Theorem 3.7]{-3}. 
\end{remark}


\begin{theorem}\label{thm:11}
Let $G$ be an edge-signed graph with $\lambda_{1}(G) > -2$, and let $v\in V(G)$. Then
$\lambda_{1}(G)>\lambda_{1}(\hat{A}(G,v))$.
Furthermore, 
$\lambda_{1}(\hat{A}(G,v))>-2$
if and only if the underlying graph of 
$G$ is the line graph of a tree $T$ and $v$ corresponds to an end-edge
of $T$.
Otherwise, 
$\lambda_{1}(\hat{A}(G,v)) \leqslant -2$.
\end{theorem}
\begin{proof}
	By Theorem~\ref{thm:CGSS}, $G$ is represented by $D_n$ or $E_8$.
	We may assume that $G$ is represented by $D_n$, otherwise by Corollary~\ref{cor:integralReps} and 		Theorem~\ref{thm:E8} we would have $\lambda_{1}(\hat{A}(G,v)) < -2$ in which case the theorem holds.
	Therefore the structure of $G$ can be described by Theorem~\ref{thm:3}.

	First suppose $G$ is of type (i) from Theorem~\ref{thm:3}.
	Suppose $G$ is the line graph of a tree $T$.
	If $v$ is not an end-edge of $T$ then $\hat{A}(G,v)$ contains $\hat{A}(\mathcal X_3^{(1)},e_3)$ as a principal submatrix, hence $\lambda_{1}(\hat{A}(G,v))\leqslant-2$.
If $v$ is an end-edge of $T$, then
$\lambda_{1}(G)>\lambda_{1}(\hat{A}(G,v))$ by Theorem~\ref{thm:Hoffman},
and $\lambda_{1}(\hat{A}(G,v))>-2$ by Lemma~\ref{lem:v1zero}.
	Next suppose $G$ is of type (i) but not the line graph of a tree $T$.
	That is, $G$ is the line graph of a unicyclic graph with an odd cycle (and $G$ is not equal to $C_3$).
	Then $\hat{A}(G,v)$ contains (as a principal submatrix) either $\hat{A}(\mathcal X_3^{(1)},e_3)$ or $\hat{A}(\mathcal X_{k,l}^{(2)},f_l)$ for some $k$ and $l$.
	
	Suppose $G$ is of type (ii).
	Then $\hat{A}(G,v)$ contains (as a principal submatrix) either $\hat{A}(\mathcal X_3^{(1)},e_3)$ or $\hat{A}(\mathcal X_{k,l}^{(3)},f_l)$ for some $k$ and $l$.
	
	Suppose $G$ is of type (iii).
	Then $G$ contains $\mathcal X_n^{(1)}$ for some $n$.
Therefore, by Lemma~\ref{lem:gtminus22}, we have
$\lambda_{1}(\hat{A}(G,v)) \leqslant -2$ for these cases.
\end{proof}


\begin{remark}
A special case of Theorem~\ref{thm:11} for unsigned graphs is given
in \cite[Theorem 5.2]{gt-3}.
\end{remark}

\section{Exceptional graphs}\label{sec:ex.grph}

In this section we enumerate the exceptional edge-signed graphs 
with smallest eigenvalue greater than $-2$, 
i.e., those that are not integrally represented.
In the tables in the appendix we list (up to switching) every such edge-signed graph.
We generalise the following result about graphs with smallest eigenvalue greater than $-2$.

\begin{theorem}[\cite{BusseNeum,DoobCvetkovic:LAA79}]\label{thm:CvetDoob}
	Let $G$ be an exceptional graph having smallest eigenvalue greater than $-2$.
	Then $G$ is one of
	\begin{enumerate}
		\item $20$ graphs on $6$ vertices;
		\item $110$ graphs on $7$ vertices;
		\item $443$ graphs on $8$ vertices.
	\end{enumerate}
\end{theorem}

\begin{figure}
\begin{center}
	\begin{tikzpicture}[yshift=4cm, xshift=0.8cm]
		\foreach \type/\pos/\name in {{vertex/(0,0.5)/bgn}, {vertex/(1,0.5)/b1}, {vertex/(2,0.5)/e1}, {vertex/(3,0.5)/b11}, {vertex/(4,0.5)/c11}, {vertex/(4,1.5)/c1}, {vertex/(5,0.5)/d1}, {vertex/(6,0.5)/f1}}
			\node[\type] (\name) at \pos {};
		\foreach \pos/\name in {{(0,0.1)/\alpha_8}, {(1,0.1)/\alpha_7}, {(2,0.1)/\alpha_6}, {(3,0.1)/\alpha_5}, {(4,0.1)/\alpha_4}, {(4,1.8)/\alpha_2}, {(5,0.1)/\alpha_3}, {(6,0.1)/\alpha_1}}
			\node at \pos {$\name$};
		\foreach \edgetype/\source/ \dest in {pedge/bgn/b1, pedge/b1/e1, pedge/e1/b11, pedge/b11/c11, pedge/c11/c1, pedge/c11/d1, pedge/d1/f1}
			\path[\edgetype] (\source) -- (\dest);
	\end{tikzpicture}
\end{center}
	\caption{The simple roots of $E_8$}
	\label{fig:E8}
\end{figure}
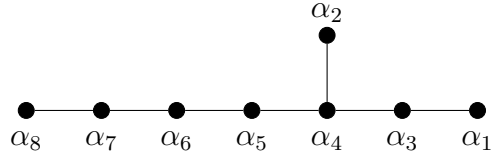


To describe our results, we need a list of $120$ lines of the
root system $E_8$. Such a list can be found in 
\cite[Appendix B]{Gleitz:KNSconj}, and this is also sufficient
to describe our results for $E_6$ and $E_7$, since these root systems
can be embedded in $E_8$. Each of the $120$ lines are determined 
by a vector $\beta=\sum_{i=1}^8 b_i\alpha_i$ and
the coefficients $(b_1,\dots,b_8)$ for each $\beta$ are given in \cite[Appendix B]{Gleitz:KNSconj}.
The inner products among the basis vectors $\alpha_1,\dots,\alpha_8$
are described by Figure~\ref{fig:E8}, where $(\alpha_i,\alpha_i)=2$
for all $i\in\{1,\dots,8\}$, $(\alpha_i,\alpha_j)=-1$ if
$\alpha_i$ and $\alpha_j$ are adjacent, 
$(\alpha_i,\alpha_j)=0$ otherwise.
The lines determined by $E_6$ are precisely the $36$ lines with
$b_7=b_8=0$, and the lines determined by $E_7$ are precisely 
the $63$ lines with $b_8=0$.

\begin{remark}
The numbering of the $120$ lines of $E_8$ in \cite[Appendix B]{Gleitz:KNSconj}
is the natural one in the following sense.
Every line can be represented by a vector
$\alpha=\sum_{i=1}^8 a_i\alpha_i$ with $\|\alpha\|^2=2$ and $a_i\geqslant0$
for all $i$.
We assign a total ordering to the lines as follows.
Let $\alpha=\sum_{i=1}^8 a_i\alpha_i$,
$\beta=\sum_{i=1}^8 b_i\alpha_i$.
We define $[\alpha] \prec [\beta]$ if either
\begin{itemize}
	\item[] $\sum_{i=1}^8 a_i < \sum_{i=1}^8 b_i$;
	\item[or]
	\item[] 
$\sum_{i=1}^8 a_i = \sum_{i=1}^8 b_i$ and
	\item[] 
$a_1=b_1,\dots,a_i=b_i,a_{i+1}>b_{i+1}$. 
\end{itemize}
Note that this ordering on the set of lines is the default ordering given by 
\texttt{MAGMA}~\cite{Magma} on the set of positive roots of the root system $E_8$ (which are in
one-to-one correspondence with the lines of the line system $E_8$).
\end{remark}

Let $n=7$ or $8$ and let $G$ be an $n$-vertex exceptional edge-signed graph.
By \cite[Theorem 1.10]{ArsComb1993}, $G$ can be obtained from an $(n-1)$-vertex exceptional edge-signed graph $H$ by attaching a vertex to $H$.

In the tables in the appendix we describe, up to switching equivalence, the edge-signed graphs that are not integrally represented.
Each edge-signed graph is described by referring to the lines used to construct it and each line is referred to by its number in \cite[Appendix B]{Gleitz:KNSconj},
or equivalently, by its position in the total ordering.
Clearly, exceptional edge-signed graphs with smallest eigenvalue greater than $-2$ must have at least $6$ vertices and at most $8$ vertices.
In Table~\ref{swc:6} the $32$ exceptional switching classes $\cS_{6,i}$
$(1\leqslant i\leqslant 32)$ of $6$-vertex edge-signed graphs are described.
The first $20$ out of the $32$ consist of those classes which contain an
unsigned graph, and these ordered according to 
\cite[Table A2]{CRS:LMS314}, in which graphs are ordered by the number of 
edges. The remaining $12$ switching classes are also ordered by the number 
of edges.
In Table~\ref{swc:7} the $233$ exceptional switching classes 
$\cS_{7,i}$ $(1\leqslant i\leqslant 233)$ of $7$-vertex 
edge-signed graphs are described.
Each switching class $\cS_{7,i}$ is obtained by adding a line 
$ l $ to $\cS_{6,k} $.
In Table~\ref{swc:7}, the triples $i, l, k$ are listed, and the first
$110$ switching classes consist of those classes which contain an
unsigned graph, and these ordered according to 
\cite[Table A2]{CRS:LMS314}.
Similarly, in Tables~\ref{swc:8_1}, \ref{swc:8_2}, \ref{swc:8_3}, \ref{swc:8_4}, and \ref{swc:8_5},
the $ 1242 $ exceptional switching classes $\cS_{8,i}$ 
$(1\leqslant i\leqslant 1242)$ of $8$-vertex edge-signed graphs are described.
Each switching class $\cS_{8,i}$ is obtained by adding a line $ l $ to 
$\cS_{7,k}$ and these triple $i,k,l$ are given in the tables.
As before, in the tables the $443$ unsigned graphs in \cite[Table A2]{CRS:LMS314} come first.

	In Tables~\ref{swc:7}--\ref{swc:8_5}, the columns labelled by $i^\prime$ denote the last two digits of $i$, that is, $i^\prime \equiv i \pmod{100}$.

We can summarise the tables below in the following theorem.

\begin{theorem}\label{thm:E8enum}
	Let $G$ be an exceptional edge-signed graph having smallest eigenvalue greater than $-2$.
	Then $G$ is one of
	\begin{enumerate}
		\item $32$ edge-signed graphs on $6$ vertices;
		\item $233$ edge-signed graphs on $7$ vertices;
		\item $1242$ edge-signed graphs on $8$ vertices.
	\end{enumerate}
\end{theorem}

\section{Hoffman graphs}
\label{sec:Hoffgraph}

A \textbf{Hoffman graph} $\mathfrak H$ is defined as a graph $(V,E)$ 
with a distinguished 
coclique 
$V_f(\mathfrak H) \subset V$ called \textbf{fat} vertices, 
the remaining vertices $V_s(\mathfrak H) = V\backslash V_f(\mathfrak H)$ 
are called \textbf{slim} vertices.
For more background on Hoffman graphs see some of 
the authors' previous papers \cite{-3,-1-tau,gt-3}.
Let $\mathfrak H$ be a Hoffman graph and suppose its adjacency matrix $A$ has 
the following form
\[
	A = \begin{pmatrix}
		A_s & C \\
		C^\top & 0
	\end{pmatrix},
\]
where the fat vertices come last.
Define $B(\mathfrak H) := A_s - C C^\top$.
The eigenvalues of $\mathfrak H$ are defined to be the eigenvalues of $B(\mathfrak H)$.
A Hoffman graph $\mathfrak H$ is called \textbf{fat} 
if every slim vertex of $\mathfrak H$ has at least one fat neighbour.
In this section we show how our results relate to the classification 
of fat Hoffman graphs with smallest eigenvalue greater than $-3$.

It is shown in \cite{gt-3} that if $\lambda_1(B(\mathfrak H)) > -3$ then 
every slim vertex is adjacent to at most two fat vertices 
and at most one slim vertex can be adjacent to more than one fat vertex.
It therefore follows that if a fat Hoffman graph $\mathfrak H$ has smallest
eigenvalue $\lambda_1(\mathfrak H) > -3$ then the diagonal of 
$B(\mathfrak H) + I$ consists of at most one $-1$ entry and the remaining entries $0$.
In other words, $B(\mathfrak H) + I$ is either the adjacency matrix $A(G)$ of a signed
graph $G$, or the modified adjacency matrix $\hat{A}(G,v)$ of $G$ with respect to
some vertex $v$.
The \textbf{special graph} $\cS(\mathfrak H)$
of a Hoffman graph $\mathfrak H$ is the edge-signed graph 
whose adjacency matrix 
has the same off-diagonal entries as $B(\mathfrak H)$ 
and zeros on the diagonal.

\begin{theorem}
Let $\mathfrak H$ be a fat Hoffman graph in which every slim vertex 
has exactly one fat neighbour. 
Then $\mathfrak H$ has smallest
eigenvalue greater than $-3$ if and only if 
$\cS(\mathfrak H)$ is switching equivalent to one of the edge-signed
graphs in Theorem~\ref{thm:3} or Theorem~\ref{thm:E8enum}.
\end{theorem}
\begin{proof}
Since every slim vertex has exactly one fat neighbour,
$B(\mathfrak H)+I$ is the adjacency matrix of $\cS(\mathfrak H)$.
Thus $\mathfrak H$ has smallest eigenvalue greater than $-3$ if and only if
$\cS(\mathfrak H)$ has smallest eigenvalue greater than $-2$.
The result then follows since Theorems~\ref{thm:3} and \ref{thm:E8enum}
give a classification of edge-signed graphs with smallest eigenvalues
greater than $-2$.
\end{proof}

\begin{lemma}\label{lm:012}
Let $\mathfrak H$ be a fat Hoffman graph 
in which every slim vertex has exactly one fat neighbour. 
Then in the special graph $\cS(\mathfrak H)$ of $\mathfrak H$,
there are no $(+)$-edges in the neighbourhood of any fat vertex.
In particular, $\cS(\mathfrak H)$ does not contain 
a cycle in which all but one edge are $(-)$-edges.
\end{lemma}
\begin{proof}
Let $N$ be the set of neighbours of a fat vertex. Then the off-diagonal
entry of $B(\mathfrak H)$ corresponding to two vertices of $N$ cannot be $1$.
This shows the first claim. 
Suppose that $\cS(\mathfrak H)$ contains a cycle
$v_0,v_1,\dots,v_n,v_0$ such that all but the edge $\{v_n,v_0\}$ are
$(-)$-edges. Then $v_i,v_{i+1}$ have a common fat neighbour for 
$i=0,1,\dots,n-1$. Since every slim vertex has exactly one fat neighbour,
it follows that $v_0,\dots,v_n$ have a common fat neighbour. But this
contradicts the first claim.
\end{proof}

Lemma~\ref{lm:012} implies that not every edge-signed graph can be the 
special graph of a fat Hoffman graph in which every slim vertex has 
exactly one fat neighbour.

For an edge-signed graph $\cS=(V,E^+,E^-)$, we denote by $\cS^-$
the unsigned graph $(V,E^-)$.
Let $\cS$ be an edge-signed graph that is switching
equivalent to one of the edge-signed
graphs in Theorem~\ref{thm:3} or Theorem~\ref{thm:E8enum}, and 
$\cS$ has no cycle in which all but one edge are $(-)$-edges.
Then every fat Hoffman graph $\mathfrak H$ in which every slim vertex has 
exactly one fat neighbour,
satisfying $\cS(\mathfrak H)=\cS$ is obtained as follows.
Let
\[
V(\cS)=\bigcup_{i=1}^n V_i\quad\text{(disjoint)}
\]
be a partition satisfying
\begin{enumerate}
\item 
for all $i\in\{1,\dots,n\}$, the subgraph induced on $V_i$ 
contains no $(+)$-edges, 
\item 
every connected component of $\cS^-$ is contained in $V_i$ for
some $i\in\{1,\dots,n\}$.
\end{enumerate}
Define a Hoffman graph $\mathfrak H$
as follows. 
The vertex set of $\mathfrak H$ consists of the set of
slim vertices $V(\cS)$ and the set of
fat vertices $\{f_1,\dots,f_n\}$. 
The edges are $\{f_i,u\}$ with $u\in V_i$, $1\leq i\leq n$;
$\{u,v\}$ with $u,v\in V_i$, $1\leq i\leq n$, $\{u,v\}\notin E(\cS)$; and
$\{u,v\}$ with $u\in V_i$, $v\in V_j$, $1\leq i,j\leq n$, 
$i\neq j$, $\{u,v\}\in E(\cS)$.
Observe that for $u,v\in V_i$, $\{u,v\}\notin E(\cS)$ if and only
if $\{u,v\}\notin E^-(\cS)$, and
for $u\in V_i$, $v\in V_j$ with $i\neq j$, $\{u,v\}\in E(\cS)$ if and only
if $\{u,v\}\in E^+(\cS)$.
Then $\cS(\mathfrak H)=\cS$ holds, 
and $\lambda_{1}(\mathfrak H)=\lambda_{1}(\cS)-1$.

In general, given an edge-signed graph $\cS$, there are a number of
fat Hoffman graphs $\mathfrak H$ with $\cS(\mathfrak H)=\cS$, 
satisfying the condition
that every slim vertex of $\mathfrak H$ has exactly one fat neighbour. 
For example,
consider the simplest case where $\cS$ is a path consisting only of
$(+)$-edges. Then such Hoffman graphs are in one-to-one correspondence
with partitions of the vertex-set into cocliques.


		\section{Acknowledgement} 
	\label{sec:acknowledgement}
		The authors are grateful for the comments of the referees.



\newpage
\section*{Appendix}

\begin{table}[h]
\caption{Switching classes $\cS_{6,i}\ (i=1,\ldots,32) $}


\label{swc:6}
\end{table}


\begin{table}
\caption{Switching classes $\cS_{7,i}\ (1 \leqslant i\leqslant 233)$}
{\scriptsize 
\begin{tabular}{|c||c|c|}
\multicolumn{3}{c}{$ 1  \leqslant i\leqslant 50 $}\\
\hline
{\bf $i'$}&{\bf $k$}&{\bf $l$}\\
\hline
\hline
$ 1 $& $ 1 $& $ 80 $ \\
\hline
$ 2 $& $ 1 $& $ 75 $ \\
\hline
$ 3 $& $ 2 $& $ 58 $ \\
\hline
$ 4 $& $ 1 $& $ 55 $ \\
\hline
$ 5 $& $ 3 $& $ 7 $ \\
\hline
$ 6 $& $ 1 $& $ 59 $ \\
\hline
$ 7 $& $ 2 $& $ 89 $ \\
\hline
$ 8 $& $ 3 $& $ 70 $ \\
\hline
$ 9 $& $ 1 $& $ 85 $ \\
\hline
$ 10 $& $ 3 $& $ 80 $ \\
\hline
$ 11 $& $ 3 $& $ 14 $ \\
\hline
$ 12 $& $ 4 $& $ 7 $ \\
\hline
$ 13 $& $ 3 $& $ 34 $ \\
\hline
$ 14 $& $ 6 $& $ 97 $ \\
\hline
$ 15 $& $ 1 $& $ 64 $ \\
\hline
$ 16 $& $ 6 $& $ 21 $ \\
\hline
$ 17 $& $ 2 $& $ 71 $ \\
\hline
$ 18 $& $ 2 $& $ 28 $ \\
\hline
$ 19 $& $ 6 $& $ 61 $ \\
\hline
$ 20 $& $ 7 $& $ 7 $ \\
\hline
$ 21 $& $ 4 $& $ 75 $ \\
\hline
$ 22 $& $ 6 $& $ 93 $ \\
\hline
$ 23 $& $ 6 $& $ 14 $ \\
\hline
$ 24 $& $ 8 $& $ 7 $ \\
\hline
$ 25 $& $ 4 $& $ 64 $ \\
\hline
$ 26 $& $ 9 $& $ 97 $ \\
\hline
$ 27 $& $ 1 $& $ 41 $ \\
\hline
$ 28 $& $ 3 $& $ 58 $ \\
\hline
$ 29 $& $ 8 $& $ 21 $ \\
\hline
$ 30 $& $ 3 $& $ 93 $ \\
\hline
$ 31 $& $ 7 $& $ 76 $ \\
\hline
$ 32 $& $ 4 $& $ 93 $ \\
\hline
$ 33 $& $ 7 $& $ 93 $ \\
\hline
$ 34 $& $ 8 $& $ 55 $ \\
\hline
$ 35 $& $ 4 $& $ 85 $ \\
\hline
$ 36 $& $ 6 $& $ 53 $ \\
\hline
$ 37 $& $ 11 $& $ 80 $ \\
\hline
$ 38 $& $ 5 $& $ 28 $ \\
\hline
$ 39 $& $ 11 $& $ 97 $ \\
\hline
$ 40 $& $ 8 $& $ 97 $ \\
\hline
$ 41 $& $ 12 $& $ 80 $ \\
\hline
$ 42 $& $ 5 $& $ 93 $ \\
\hline
$ 43 $& $ 3 $& $ 41 $ \\
\hline
$ 44 $& $ 4 $& $ 71 $ \\
\hline
$ 45 $& $ 13 $& $ 97 $ \\
\hline
$ 46 $& $ 13 $& $ 7 $ \\
\hline
$ 47 $& $ 7 $& $ 28 $ \\
\hline
$ 48 $& $ 7 $& $ 53 $ \\
\hline
$ 49 $& $ 10 $& $ 97 $ \\
\hline
$ 50 $& $ 8 $& $ 53 $ \\
\hline
\end{tabular}
\begin{tabular}{|c||c|c|}
\multicolumn{3}{c}{$ 51  \leqslant i\leqslant 100 $}\\
\hline
{\bf $i'$}&{\bf $k$}&{\bf $l$}\\
\hline
\hline
$ 51 $& $ 8 $& $ 61 $ \\
\hline
$ 52 $& $ 7 $& $ 97 $ \\
\hline
$ 53 $& $ 9 $& $ 59 $ \\
\hline
$ 54 $& $ 2 $& $ 41 $ \\
\hline
$ 55 $& $ 9 $& $ 58 $ \\
\hline
$ 56 $& $ 12 $& $ 76 $ \\
\hline
$ 57 $& $ 5 $& $ 41 $ \\
\hline
$ 58 $& $ 8 $& $ 76 $ \\
\hline
$ 59 $& $ 8 $& $ 66 $ \\
\hline
$ 60 $& $ 9 $& $ 28 $ \\
\hline
$ 61 $& $ 6 $& $ 70 $ \\
\hline
$ 62 $& $ 6 $& $ 71 $ \\
\hline
$ 63 $& $ 13 $& $ 80 $ \\
\hline
$ 64 $& $ 16 $& $ 97 $ \\
\hline
$ 65 $& $ 6 $& $ 28 $ \\
\hline
$ 66 $& $ 10 $& $ 59 $ \\
\hline
$ 67 $& $ 12 $& $ 28 $ \\
\hline
$ 68 $& $ 13 $& $ 89 $ \\
\hline
$ 69 $& $ 8 $& $ 93 $ \\
\hline
$ 70 $& $ 7 $& $ 59 $ \\
\hline
$ 71 $& $ 9 $& $ 41 $ \\
\hline
$ 72 $& $ 9 $& $ 53 $ \\
\hline
$ 73 $& $ 15 $& $ 34 $ \\
\hline
$ 74 $& $ 16 $& $ 93 $ \\
\hline
$ 75 $& $ 17 $& $ 75 $ \\
\hline
$ 76 $& $ 12 $& $ 97 $ \\
\hline
$ 77 $& $ 18 $& $ 7 $ \\
\hline
$ 78 $& $ 6 $& $ 41 $ \\
\hline
$ 79 $& $ 18 $& $ 70 $ \\
\hline
$ 80 $& $ 14 $& $ 85 $ \\
\hline
$ 81 $& $ 10 $& $ 28 $ \\
\hline
$ 82 $& $ 18 $& $ 85 $ \\
\hline
$ 83 $& $ 15 $& $ 28 $ \\
\hline
$ 84 $& $ 18 $& $ 71 $ \\
\hline
$ 85 $& $ 8 $& $ 46 $ \\
\hline
$ 86 $& $ 12 $& $ 53 $ \\
\hline
$ 87 $& $ 8 $& $ 35 $ \\
\hline
$ 88 $& $ 13 $& $ 28 $ \\
\hline
$ 89 $& $ 12 $& $ 59 $ \\
\hline
$ 90 $& $ 13 $& $ 93 $ \\
\hline
$ 91 $& $ 14 $& $ 28 $ \\
\hline
$ 92 $& $ 18 $& $ 21 $ \\
\hline
$ 93 $& $ 11 $& $ 41 $ \\
\hline
$ 94 $& $ 19 $& $ 89 $ \\
\hline
$ 95 $& $ 13 $& $ 59 $ \\
\hline
$ 96 $& $ 20 $& $ 71 $ \\
\hline
$ 97 $& $ 20 $& $ 75 $ \\
\hline
$ 98 $& $ 18 $& $ 49 $ \\
\hline
$ 99 $& $ 19 $& $ 58 $ \\
\hline
$ 0 $& $ 17 $& $ 66 $ \\
\hline
\end{tabular}
\begin{tabular}{|c||c|c|}
\multicolumn{3}{c}{$ 101  \leqslant i\leqslant 150 $}\\
\hline
{\bf $i'$}&{\bf $k$}&{\bf $l$}\\
\hline
\hline
$ 1 $& $ 12 $& $ 41 $ \\
\hline
$ 2 $& $ 15 $& $ 59 $ \\
\hline
$ 3 $& $ 16 $& $ 28 $ \\
\hline
$ 4 $& $ 19 $& $ 46 $ \\
\hline
$ 5 $& $ 20 $& $ 85 $ \\
\hline
$ 6 $& $ 16 $& $ 41 $ \\
\hline
$ 7 $& $ 20 $& $ 41 $ \\
\hline
$ 8 $& $ 19 $& $ 35 $ \\
\hline
$ 9 $& $ 20 $& $ 66 $ \\
\hline
$ 10 $& $ 20 $& $ 59 $ \\
\hline
$ 11 $& $ 1 $& $ 82 $ \\
\hline
$ 12 $& $ 1 $& $ 49 $ \\
\hline
$ 13 $& $ 1 $& $ 97 $ \\
\hline
$ 14 $& $ 3 $& $ 46 $ \\
\hline
$ 15 $& $ 23 $& $ 97 $ \\
\hline
$ 16 $& $ 21 $& $ 58 $ \\
\hline
$ 17 $& $ 21 $& $ 75 $ \\
\hline
$ 18 $& $ 2 $& $ 75 $ \\
\hline
$ 19 $& $ 23 $& $ 82 $ \\
\hline
$ 20 $& $ 24 $& $ 21 $ \\
\hline
$ 21 $& $ 1 $& $ 89 $ \\
\hline
$ 22 $& $ 23 $& $ 7 $ \\
\hline
$ 23 $& $ 21 $& $ 49 $ \\
\hline
$ 24 $& $ 4 $& $ 14 $ \\
\hline
$ 25 $& $ 26 $& $ 97 $ \\
\hline
$ 26 $& $ 22 $& $ 49 $ \\
\hline
$ 27 $& $ 27 $& $ 97 $ \\
\hline
$ 28 $& $ 4 $& $ 34 $ \\
\hline
$ 29 $& $ 23 $& $ 46 $ \\
\hline
$ 30 $& $ 6 $& $ 49 $ \\
\hline
$ 31 $& $ 23 $& $ 61 $ \\
\hline
$ 32 $& $ 22 $& $ 75 $ \\
\hline
$ 33 $& $ 23 $& $ 21 $ \\
\hline
$ 34 $& $ 21 $& $ 85 $ \\
\hline
$ 35 $& $ 24 $& $ 97 $ \\
\hline
$ 36 $& $ 23 $& $ 28 $ \\
\hline
$ 37 $& $ 23 $& $ 34 $ \\
\hline
$ 38 $& $ 4 $& $ 97 $ \\
\hline
$ 39 $& $ 5 $& $ 46 $ \\
\hline
$ 40 $& $ 25 $& $ 39 $ \\
\hline
$ 41 $& $ 21 $& $ 89 $ \\
\hline
$ 42 $& $ 3 $& $ 39 $ \\
\hline
$ 43 $& $ 3 $& $ 76 $ \\
\hline
$ 44 $& $ 23 $& $ 80 $ \\
\hline
$ 45 $& $ 24 $& $ 80 $ \\
\hline
$ 46 $& $ 25 $& $ 89 $ \\
\hline
$ 47 $& $ 23 $& $ 14 $ \\
\hline
$ 48 $& $ 24 $& $ 85 $ \\
\hline
$ 49 $& $ 24 $& $ 71 $ \\
\hline
$ 50 $& $ 25 $& $ 21 $ \\
\hline
\end{tabular}
\begin{tabular}{|c||c|c|}
\multicolumn{3}{c}{$ 151  \leqslant i\leqslant 200 $}\\
\hline
{\bf $i'$}&{\bf $k$}&{\bf $l$}\\
\hline
\hline
$ 51 $& $ 25 $& $ 70 $ \\
\hline
$ 52 $& $ 5 $& $ 85 $ \\
\hline
$ 53 $& $ 26 $& $ 76 $ \\
\hline
$ 54 $& $ 27 $& $ 61 $ \\
\hline
$ 55 $& $ 24 $& $ 66 $ \\
\hline
$ 56 $& $ 28 $& $ 14 $ \\
\hline
$ 57 $& $ 6 $& $ 76 $ \\
\hline
$ 58 $& $ 8 $& $ 89 $ \\
\hline
$ 59 $& $ 23 $& $ 41 $ \\
\hline
$ 60 $& $ 29 $& $ 7 $ \\
\hline
$ 61 $& $ 28 $& $ 97 $ \\
\hline
$ 62 $& $ 29 $& $ 21 $ \\
\hline
$ 63 $& $ 23 $& $ 89 $ \\
\hline
$ 64 $& $ 23 $& $ 58 $ \\
\hline
$ 65 $& $ 26 $& $ 49 $ \\
\hline
$ 66 $& $ 25 $& $ 75 $ \\
\hline
$ 67 $& $ 21 $& $ 59 $ \\
\hline
$ 68 $& $ 25 $& $ 97 $ \\
\hline
$ 69 $& $ 13 $& $ 76 $ \\
\hline
$ 70 $& $ 25 $& $ 61 $ \\
\hline
$ 71 $& $ 25 $& $ 71 $ \\
\hline
$ 72 $& $ 4 $& $ 89 $ \\
\hline
$ 73 $& $ 29 $& $ 55 $ \\
\hline
$ 74 $& $ 9 $& $ 46 $ \\
\hline
$ 75 $& $ 28 $& $ 55 $ \\
\hline
$ 76 $& $ 27 $& $ 64 $ \\
\hline
$ 77 $& $ 24 $& $ 64 $ \\
\hline
$ 78 $& $ 26 $& $ 66 $ \\
\hline
$ 79 $& $ 25 $& $ 55 $ \\
\hline
$ 80 $& $ 28 $& $ 59 $ \\
\hline
$ 81 $& $ 27 $& $ 39 $ \\
\hline
$ 82 $& $ 8 $& $ 58 $ \\
\hline
$ 83 $& $ 13 $& $ 46 $ \\
\hline
$ 84 $& $ 5 $& $ 39 $ \\
\hline
$ 85 $& $ 24 $& $ 76 $ \\
\hline
$ 86 $& $ 27 $& $ 71 $ \\
\hline
$ 87 $& $ 25 $& $ 66 $ \\
\hline
$ 88 $& $ 29 $& $ 82 $ \\
\hline
$ 89 $& $ 8 $& $ 39 $ \\
\hline
$ 90 $& $ 24 $& $ 41 $ \\
\hline
$ 91 $& $ 23 $& $ 71 $ \\
\hline
$ 92 $& $ 28 $& $ 80 $ \\
\hline
$ 93 $& $ 30 $& $ 21 $ \\
\hline
$ 94 $& $ 11 $& $ 49 $ \\
\hline
$ 95 $& $ 22 $& $ 46 $ \\
\hline
$ 96 $& $ 25 $& $ 53 $ \\
\hline
$ 97 $& $ 13 $& $ 49 $ \\
\hline
$ 98 $& $ 31 $& $ 97 $ \\
\hline
$ 99 $& $ 28 $& $ 89 $ \\
\hline
$ 0 $& $ 12 $& $ 49 $ \\
\hline
\end{tabular}
\begin{tabular}{|c||c|c|}
\multicolumn{3}{c}{$ 201  \leqslant i\leqslant 233 $}\\
\hline
{\bf $i'$}&{\bf $k$}&{\bf $l$}\\
\hline
\hline
$ 1 $& $ 16 $& $ 49 $ \\
\hline
$ 2 $& $ 26 $& $ 41 $ \\
\hline
$ 3 $& $ 29 $& $ 93 $ \\
\hline
$ 4 $& $ 30 $& $ 75 $ \\
\hline
$ 5 $& $ 16 $& $ 34 $ \\
\hline
$ 6 $& $ 7 $& $ 49 $ \\
\hline
\end{tabular}

}
\label{swc:7}
\end{table}


\begin{table}
\caption{Switching classes $\cS_{8,i}\ (1 \leqslant i\leqslant 250) $}
{\tiny
\begin{tabular}{|c||c|c|}
\multicolumn{3}{c}{$ 1  \leqslant i\leqslant 50 $}\\
\hline
{\bf $i'$}&{\bf $k$}&{\bf $l$}\\
\hline
\hline
$ 1 $& $ 1 $& $ 8 $ \\
\hline
$ 2 $& $ 1 $& $ 118 $ \\
\hline
$ 3 $& $ 2 $& $ 118 $ \\
\hline
$ 4 $& $ 1 $& $ 96 $ \\
\hline
$ 5 $& $ 1 $& $ 92 $ \\
\hline
$ 6 $& $ 5 $& $ 8 $ \\
\hline
$ 7 $& $ 1 $& $ 95 $ \\
\hline
$ 8 $& $ 6 $& $ 118 $ \\
\hline
$ 9 $& $ 6 $& $ 8 $ \\
\hline
$ 10 $& $ 4 $& $ 96 $ \\
\hline
$ 11 $& $ 3 $& $ 106 $ \\
\hline
$ 12 $& $ 7 $& $ 8 $ \\
\hline
$ 13 $& $ 8 $& $ 74 $ \\
\hline
$ 14 $& $ 1 $& $ 90 $ \\
\hline
$ 15 $& $ 9 $& $ 8 $ \\
\hline
$ 16 $& $ 10 $& $ 74 $ \\
\hline
$ 17 $& $ 5 $& $ 118 $ \\
\hline
$ 18 $& $ 5 $& $ 117 $ \\
\hline
$ 19 $& $ 12 $& $ 8 $ \\
\hline
$ 20 $& $ 6 $& $ 62 $ \\
\hline
$ 21 $& $ 14 $& $ 8 $ \\
\hline
$ 22 $& $ 15 $& $ 118 $ \\
\hline
$ 23 $& $ 5 $& $ 42 $ \\
\hline
$ 24 $& $ 13 $& $ 8 $ \\
\hline
$ 25 $& $ 16 $& $ 118 $ \\
\hline
$ 26 $& $ 9 $& $ 96 $ \\
\hline
$ 27 $& $ 3 $& $ 78 $ \\
\hline
$ 28 $& $ 9 $& $ 62 $ \\
\hline
$ 29 $& $ 7 $& $ 74 $ \\
\hline
$ 30 $& $ 11 $& $ 74 $ \\
\hline
$ 31 $& $ 12 $& $ 68 $ \\
\hline
$ 32 $& $ 7 $& $ 90 $ \\
\hline
$ 33 $& $ 4 $& $ 117 $ \\
\hline
$ 34 $& $ 9 $& $ 78 $ \\
\hline
$ 35 $& $ 19 $& $ 8 $ \\
\hline
$ 36 $& $ 10 $& $ 42 $ \\
\hline
$ 37 $& $ 20 $& $ 8 $ \\
\hline
$ 38 $& $ 3 $& $ 108 $ \\
\hline
$ 39 $& $ 14 $& $ 116 $ \\
\hline
$ 40 $& $ 20 $& $ 62 $ \\
\hline
$ 41 $& $ 12 $& $ 118 $ \\
\hline
$ 42 $& $ 14 $& $ 74 $ \\
\hline
$ 43 $& $ 4 $& $ 103 $ \\
\hline
$ 44 $& $ 13 $& $ 117 $ \\
\hline
$ 45 $& $ 12 $& $ 96 $ \\
\hline
$ 46 $& $ 22 $& $ 74 $ \\
\hline
$ 47 $& $ 1 $& $ 79 $ \\
\hline
$ 48 $& $ 6 $& $ 103 $ \\
\hline
$ 49 $& $ 5 $& $ 108 $ \\
\hline
$ 50 $& $ 12 $& $ 103 $ \\
\hline
\end{tabular}
\begin{tabular}{|c||c|c|}
\multicolumn{3}{c}{$ 51  \leqslant i\leqslant 100 $}\\
\hline
{\bf $i'$}&{\bf $k$}&{\bf $l$}\\
\hline
\hline
$ 51 $& $ 14 $& $ 90 $ \\
\hline
$ 52 $& $ 16 $& $ 116 $ \\
\hline
$ 53 $& $ 5 $& $ 106 $ \\
\hline
$ 54 $& $ 1 $& $ 50 $ \\
\hline
$ 55 $& $ 28 $& $ 74 $ \\
\hline
$ 56 $& $ 14 $& $ 106 $ \\
\hline
$ 57 $& $ 6 $& $ 102 $ \\
\hline
$ 58 $& $ 27 $& $ 8 $ \\
\hline
$ 59 $& $ 4 $& $ 62 $ \\
\hline
$ 60 $& $ 6 $& $ 105 $ \\
\hline
$ 61 $& $ 17 $& $ 74 $ \\
\hline
$ 62 $& $ 9 $& $ 117 $ \\
\hline
$ 63 $& $ 8 $& $ 77 $ \\
\hline
$ 64 $& $ 19 $& $ 74 $ \\
\hline
$ 65 $& $ 19 $& $ 118 $ \\
\hline
$ 66 $& $ 31 $& $ 8 $ \\
\hline
$ 67 $& $ 9 $& $ 103 $ \\
\hline
$ 68 $& $ 32 $& $ 68 $ \\
\hline
$ 69 $& $ 32 $& $ 8 $ \\
\hline
$ 70 $& $ 19 $& $ 117 $ \\
\hline
$ 71 $& $ 8 $& $ 84 $ \\
\hline
$ 72 $& $ 2 $& $ 108 $ \\
\hline
$ 73 $& $ 3 $& $ 77 $ \\
\hline
$ 74 $& $ 28 $& $ 77 $ \\
\hline
$ 75 $& $ 2 $& $ 50 $ \\
\hline
$ 76 $& $ 10 $& $ 106 $ \\
\hline
$ 77 $& $ 4 $& $ 36 $ \\
\hline
$ 78 $& $ 11 $& $ 108 $ \\
\hline
$ 79 $& $ 21 $& $ 103 $ \\
\hline
$ 80 $& $ 27 $& $ 62 $ \\
\hline
$ 81 $& $ 28 $& $ 117 $ \\
\hline
$ 82 $& $ 34 $& $ 118 $ \\
\hline
$ 83 $& $ 7 $& $ 84 $ \\
\hline
$ 84 $& $ 30 $& $ 77 $ \\
\hline
$ 85 $& $ 35 $& $ 68 $ \\
\hline
$ 86 $& $ 4 $& $ 47 $ \\
\hline
$ 87 $& $ 36 $& $ 118 $ \\
\hline
$ 88 $& $ 14 $& $ 92 $ \\
\hline
$ 89 $& $ 35 $& $ 8 $ \\
\hline
$ 90 $& $ 22 $& $ 106 $ \\
\hline
$ 91 $& $ 13 $& $ 108 $ \\
\hline
$ 92 $& $ 1 $& $ 36 $ \\
\hline
$ 93 $& $ 30 $& $ 118 $ \\
\hline
$ 94 $& $ 36 $& $ 8 $ \\
\hline
$ 95 $& $ 39 $& $ 92 $ \\
\hline
$ 96 $& $ 23 $& $ 106 $ \\
\hline
$ 97 $& $ 25 $& $ 106 $ \\
\hline
$ 98 $& $ 27 $& $ 78 $ \\
\hline
$ 99 $& $ 11 $& $ 117 $ \\
\hline
$ 0 $& $ 41 $& $ 8 $ \\
\hline
\end{tabular}
\begin{tabular}{|c||c|c|}
\multicolumn{3}{c}{$ 101  \leqslant i\leqslant 150 $}\\
\hline
{\bf $i'$}&{\bf $k$}&{\bf $l$}\\
\hline
\hline
$ 1 $& $ 5 $& $ 50 $ \\
\hline
$ 2 $& $ 42 $& $ 74 $ \\
\hline
$ 3 $& $ 44 $& $ 68 $ \\
\hline
$ 4 $& $ 14 $& $ 84 $ \\
\hline
$ 5 $& $ 44 $& $ 8 $ \\
\hline
$ 6 $& $ 43 $& $ 8 $ \\
\hline
$ 7 $& $ 45 $& $ 87 $ \\
\hline
$ 8 $& $ 20 $& $ 95 $ \\
\hline
$ 9 $& $ 32 $& $ 118 $ \\
\hline
$ 10 $& $ 21 $& $ 79 $ \\
\hline
$ 11 $& $ 19 $& $ 68 $ \\
\hline
$ 12 $& $ 24 $& $ 87 $ \\
\hline
$ 13 $& $ 22 $& $ 47 $ \\
\hline
$ 14 $& $ 25 $& $ 47 $ \\
\hline
$ 15 $& $ 10 $& $ 78 $ \\
\hline
$ 16 $& $ 48 $& $ 96 $ \\
\hline
$ 17 $& $ 37 $& $ 118 $ \\
\hline
$ 18 $& $ 17 $& $ 84 $ \\
\hline
$ 19 $& $ 9 $& $ 50 $ \\
\hline
$ 20 $& $ 39 $& $ 116 $ \\
\hline
$ 21 $& $ 40 $& $ 87 $ \\
\hline
$ 22 $& $ 21 $& $ 105 $ \\
\hline
$ 23 $& $ 37 $& $ 92 $ \\
\hline
$ 24 $& $ 23 $& $ 108 $ \\
\hline
$ 25 $& $ 12 $& $ 90 $ \\
\hline
$ 26 $& $ 5 $& $ 78 $ \\
\hline
$ 27 $& $ 25 $& $ 102 $ \\
\hline
$ 28 $& $ 50 $& $ 118 $ \\
\hline
$ 29 $& $ 35 $& $ 103 $ \\
\hline
$ 30 $& $ 9 $& $ 90 $ \\
\hline
$ 31 $& $ 25 $& $ 105 $ \\
\hline
$ 32 $& $ 37 $& $ 78 $ \\
\hline
$ 33 $& $ 16 $& $ 105 $ \\
\hline
$ 34 $& $ 26 $& $ 105 $ \\
\hline
$ 35 $& $ 10 $& $ 50 $ \\
\hline
$ 36 $& $ 2 $& $ 36 $ \\
\hline
$ 37 $& $ 12 $& $ 78 $ \\
\hline
$ 38 $& $ 53 $& $ 77 $ \\
\hline
$ 39 $& $ 40 $& $ 77 $ \\
\hline
$ 40 $& $ 25 $& $ 95 $ \\
\hline
$ 41 $& $ 55 $& $ 74 $ \\
\hline
$ 42 $& $ 55 $& $ 77 $ \\
\hline
$ 43 $& $ 41 $& $ 92 $ \\
\hline
$ 44 $& $ 21 $& $ 96 $ \\
\hline
$ 45 $& $ 22 $& $ 79 $ \\
\hline
$ 46 $& $ 43 $& $ 117 $ \\
\hline
$ 47 $& $ 41 $& $ 95 $ \\
\hline
$ 48 $& $ 23 $& $ 117 $ \\
\hline
$ 49 $& $ 44 $& $ 102 $ \\
\hline
$ 50 $& $ 43 $& $ 42 $ \\
\hline
\end{tabular}
\begin{tabular}{|c||c|c|}
\multicolumn{3}{c}{$ 151  \leqslant i\leqslant 200 $}\\
\hline
{\bf $i'$}&{\bf $k$}&{\bf $l$}\\
\hline
\hline
$ 51 $& $ 45 $& $ 74 $ \\
\hline
$ 52 $& $ 27 $& $ 95 $ \\
\hline
$ 53 $& $ 58 $& $ 8 $ \\
\hline
$ 54 $& $ 46 $& $ 118 $ \\
\hline
$ 55 $& $ 5 $& $ 105 $ \\
\hline
$ 56 $& $ 14 $& $ 77 $ \\
\hline
$ 57 $& $ 28 $& $ 79 $ \\
\hline
$ 58 $& $ 15 $& $ 102 $ \\
\hline
$ 59 $& $ 59 $& $ 8 $ \\
\hline
$ 60 $& $ 14 $& $ 78 $ \\
\hline
$ 61 $& $ 45 $& $ 103 $ \\
\hline
$ 62 $& $ 43 $& $ 74 $ \\
\hline
$ 63 $& $ 44 $& $ 103 $ \\
\hline
$ 64 $& $ 59 $& $ 116 $ \\
\hline
$ 65 $& $ 45 $& $ 92 $ \\
\hline
$ 66 $& $ 64 $& $ 8 $ \\
\hline
$ 67 $& $ 65 $& $ 118 $ \\
\hline
$ 68 $& $ 16 $& $ 36 $ \\
\hline
$ 69 $& $ 31 $& $ 77 $ \\
\hline
$ 70 $& $ 32 $& $ 90 $ \\
\hline
$ 71 $& $ 36 $& $ 68 $ \\
\hline
$ 72 $& $ 48 $& $ 62 $ \\
\hline
$ 73 $& $ 52 $& $ 116 $ \\
\hline
$ 74 $& $ 31 $& $ 95 $ \\
\hline
$ 75 $& $ 21 $& $ 47 $ \\
\hline
$ 76 $& $ 52 $& $ 95 $ \\
\hline
$ 77 $& $ 35 $& $ 105 $ \\
\hline
$ 78 $& $ 52 $& $ 78 $ \\
\hline
$ 79 $& $ 37 $& $ 77 $ \\
\hline
$ 80 $& $ 52 $& $ 92 $ \\
\hline
$ 81 $& $ 39 $& $ 68 $ \\
\hline
$ 82 $& $ 51 $& $ 118 $ \\
\hline
$ 83 $& $ 34 $& $ 42 $ \\
\hline
$ 84 $& $ 51 $& $ 62 $ \\
\hline
$ 85 $& $ 66 $& $ 96 $ \\
\hline
$ 86 $& $ 59 $& $ 87 $ \\
\hline
$ 87 $& $ 24 $& $ 95 $ \\
\hline
$ 88 $& $ 36 $& $ 90 $ \\
\hline
$ 89 $& $ 20 $& $ 74 $ \\
\hline
$ 90 $& $ 55 $& $ 79 $ \\
\hline
$ 91 $& $ 34 $& $ 62 $ \\
\hline
$ 92 $& $ 22 $& $ 77 $ \\
\hline
$ 93 $& $ 42 $& $ 103 $ \\
\hline
$ 94 $& $ 36 $& $ 79 $ \\
\hline
$ 95 $& $ 67 $& $ 118 $ \\
\hline
$ 96 $& $ 22 $& $ 78 $ \\
\hline
$ 97 $& $ 39 $& $ 95 $ \\
\hline
$ 98 $& $ 45 $& $ 84 $ \\
\hline
$ 99 $& $ 25 $& $ 108 $ \\
\hline
$ 0 $& $ 63 $& $ 42 $ \\
\hline
\end{tabular}
\begin{tabular}{|c||c|c|}
\multicolumn{3}{c}{$ 201  \leqslant i\leqslant 250 $}\\
\hline
{\bf $i'$}&{\bf $k$}&{\bf $l$}\\
\hline
\hline
$ 1 $& $ 64 $& $ 47 $ \\
\hline
$ 2 $& $ 19 $& $ 36 $ \\
\hline
$ 3 $& $ 58 $& $ 74 $ \\
\hline
$ 4 $& $ 28 $& $ 50 $ \\
\hline
$ 5 $& $ 69 $& $ 78 $ \\
\hline
$ 6 $& $ 60 $& $ 116 $ \\
\hline
$ 7 $& $ 36 $& $ 84 $ \\
\hline
$ 8 $& $ 58 $& $ 77 $ \\
\hline
$ 9 $& $ 28 $& $ 84 $ \\
\hline
$ 10 $& $ 65 $& $ 47 $ \\
\hline
$ 11 $& $ 71 $& $ 8 $ \\
\hline
$ 12 $& $ 61 $& $ 87 $ \\
\hline
$ 13 $& $ 73 $& $ 118 $ \\
\hline
$ 14 $& $ 62 $& $ 106 $ \\
\hline
$ 15 $& $ 64 $& $ 74 $ \\
\hline
$ 16 $& $ 74 $& $ 74 $ \\
\hline
$ 17 $& $ 65 $& $ 117 $ \\
\hline
$ 18 $& $ 75 $& $ 8 $ \\
\hline
$ 19 $& $ 41 $& $ 74 $ \\
\hline
$ 20 $& $ 75 $& $ 118 $ \\
\hline
$ 21 $& $ 64 $& $ 90 $ \\
\hline
$ 22 $& $ 77 $& $ 117 $ \\
\hline
$ 23 $& $ 24 $& $ 105 $ \\
\hline
$ 24 $& $ 26 $& $ 103 $ \\
\hline
$ 25 $& $ 28 $& $ 106 $ \\
\hline
$ 26 $& $ 78 $& $ 116 $ \\
\hline
$ 27 $& $ 79 $& $ 74 $ \\
\hline
$ 28 $& $ 30 $& $ 79 $ \\
\hline
$ 29 $& $ 47 $& $ 77 $ \\
\hline
$ 30 $& $ 33 $& $ 36 $ \\
\hline
$ 31 $& $ 54 $& $ 78 $ \\
\hline
$ 32 $& $ 32 $& $ 79 $ \\
\hline
$ 33 $& $ 32 $& $ 105 $ \\
\hline
$ 34 $& $ 51 $& $ 42 $ \\
\hline
$ 35 $& $ 49 $& $ 105 $ \\
\hline
$ 36 $& $ 51 $& $ 96 $ \\
\hline
$ 37 $& $ 48 $& $ 79 $ \\
\hline
$ 38 $& $ 53 $& $ 102 $ \\
\hline
$ 39 $& $ 67 $& $ 92 $ \\
\hline
$ 40 $& $ 58 $& $ 108 $ \\
\hline
$ 41 $& $ 50 $& $ 62 $ \\
\hline
$ 42 $& $ 67 $& $ 116 $ \\
\hline
$ 43 $& $ 49 $& $ 102 $ \\
\hline
$ 44 $& $ 23 $& $ 105 $ \\
\hline
$ 45 $& $ 24 $& $ 79 $ \\
\hline
$ 46 $& $ 53 $& $ 79 $ \\
\hline
$ 47 $& $ 50 $& $ 56 $ \\
\hline
$ 48 $& $ 50 $& $ 102 $ \\
\hline
$ 49 $& $ 75 $& $ 68 $ \\
\hline
$ 50 $& $ 71 $& $ 116 $ \\
\hline
\end{tabular}

}
\label{swc:8_1}
\end{table}


\begin{table}
\caption{Switching classes $\cS_{8,i}\ (251 \leqslant i\leqslant 500) $}
{\tiny
\begin{tabular}{|c||c|c|}
\multicolumn{3}{c}{$ 251  \leqslant i\leqslant 300 $}\\
\hline
{\bf $i'$}&{\bf $k$}&{\bf $l$}\\
\hline
\hline
$ 51 $& $ 7 $& $ 50 $ \\
\hline
$ 52 $& $ 76 $& $ 47 $ \\
\hline
$ 53 $& $ 53 $& $ 42 $ \\
\hline
$ 54 $& $ 37 $& $ 90 $ \\
\hline
$ 55 $& $ 54 $& $ 74 $ \\
\hline
$ 56 $& $ 29 $& $ 79 $ \\
\hline
$ 57 $& $ 56 $& $ 95 $ \\
\hline
$ 58 $& $ 57 $& $ 42 $ \\
\hline
$ 59 $& $ 76 $& $ 95 $ \\
\hline
$ 60 $& $ 75 $& $ 96 $ \\
\hline
$ 61 $& $ 74 $& $ 90 $ \\
\hline
$ 62 $& $ 11 $& $ 105 $ \\
\hline
$ 63 $& $ 60 $& $ 42 $ \\
\hline
$ 64 $& $ 64 $& $ 105 $ \\
\hline
$ 65 $& $ 78 $& $ 117 $ \\
\hline
$ 66 $& $ 36 $& $ 108 $ \\
\hline
$ 67 $& $ 82 $& $ 8 $ \\
\hline
$ 68 $& $ 22 $& $ 50 $ \\
\hline
$ 69 $& $ 43 $& $ 84 $ \\
\hline
$ 70 $& $ 38 $& $ 79 $ \\
\hline
$ 71 $& $ 84 $& $ 8 $ \\
\hline
$ 72 $& $ 78 $& $ 74 $ \\
\hline
$ 73 $& $ 61 $& $ 84 $ \\
\hline
$ 74 $& $ 58 $& $ 56 $ \\
\hline
$ 75 $& $ 27 $& $ 36 $ \\
\hline
$ 76 $& $ 27 $& $ 108 $ \\
\hline
$ 77 $& $ 85 $& $ 92 $ \\
\hline
$ 78 $& $ 63 $& $ 102 $ \\
\hline
$ 79 $& $ 86 $& $ 92 $ \\
\hline
$ 80 $& $ 64 $& $ 103 $ \\
\hline
$ 81 $& $ 86 $& $ 118 $ \\
\hline
$ 82 $& $ 26 $& $ 84 $ \\
\hline
$ 83 $& $ 87 $& $ 78 $ \\
\hline
$ 84 $& $ 88 $& $ 8 $ \\
\hline
$ 85 $& $ 88 $& $ 118 $ \\
\hline
$ 86 $& $ 42 $& $ 79 $ \\
\hline
$ 87 $& $ 89 $& $ 118 $ \\
\hline
$ 88 $& $ 90 $& $ 74 $ \\
\hline
$ 89 $& $ 70 $& $ 77 $ \\
\hline
$ 90 $& $ 66 $& $ 77 $ \\
\hline
$ 91 $& $ 61 $& $ 105 $ \\
\hline
$ 92 $& $ 81 $& $ 47 $ \\
\hline
$ 93 $& $ 69 $& $ 68 $ \\
\hline
$ 94 $& $ 55 $& $ 105 $ \\
\hline
$ 95 $& $ 67 $& $ 103 $ \\
\hline
$ 96 $& $ 67 $& $ 102 $ \\
\hline
$ 97 $& $ 71 $& $ 42 $ \\
\hline
$ 98 $& $ 51 $& $ 68 $ \\
\hline
$ 99 $& $ 81 $& $ 96 $ \\
\hline
$ 0 $& $ 68 $& $ 102 $ \\
\hline
\end{tabular}
\begin{tabular}{|c||c|c|}
\multicolumn{3}{c}{$ 301  \leqslant i\leqslant 350 $}\\
\hline
{\bf $i'$}&{\bf $k$}&{\bf $l$}\\
\hline
\hline
$ 1 $& $ 82 $& $ 68 $ \\
\hline
$ 2 $& $ 63 $& $ 78 $ \\
\hline
$ 3 $& $ 91 $& $ 62 $ \\
\hline
$ 4 $& $ 36 $& $ 78 $ \\
\hline
$ 5 $& $ 83 $& $ 42 $ \\
\hline
$ 6 $& $ 60 $& $ 108 $ \\
\hline
$ 7 $& $ 52 $& $ 105 $ \\
\hline
$ 8 $& $ 20 $& $ 50 $ \\
\hline
$ 9 $& $ 91 $& $ 118 $ \\
\hline
$ 10 $& $ 85 $& $ 74 $ \\
\hline
$ 11 $& $ 74 $& $ 79 $ \\
\hline
$ 12 $& $ 56 $& $ 108 $ \\
\hline
$ 13 $& $ 76 $& $ 78 $ \\
\hline
$ 14 $& $ 92 $& $ 118 $ \\
\hline
$ 15 $& $ 34 $& $ 43 $ \\
\hline
$ 16 $& $ 88 $& $ 42 $ \\
\hline
$ 17 $& $ 41 $& $ 105 $ \\
\hline
$ 18 $& $ 76 $& $ 68 $ \\
\hline
$ 19 $& $ 76 $& $ 102 $ \\
\hline
$ 20 $& $ 61 $& $ 36 $ \\
\hline
$ 21 $& $ 94 $& $ 74 $ \\
\hline
$ 22 $& $ 39 $& $ 108 $ \\
\hline
$ 23 $& $ 59 $& $ 105 $ \\
\hline
$ 24 $& $ 89 $& $ 103 $ \\
\hline
$ 25 $& $ 59 $& $ 102 $ \\
\hline
$ 26 $& $ 42 $& $ 105 $ \\
\hline
$ 27 $& $ 62 $& $ 84 $ \\
\hline
$ 28 $& $ 90 $& $ 118 $ \\
\hline
$ 29 $& $ 96 $& $ 8 $ \\
\hline
$ 30 $& $ 43 $& $ 50 $ \\
\hline
$ 31 $& $ 96 $& $ 68 $ \\
\hline
$ 32 $& $ 46 $& $ 50 $ \\
\hline
$ 33 $& $ 68 $& $ 106 $ \\
\hline
$ 34 $& $ 50 $& $ 110 $ \\
\hline
$ 35 $& $ 81 $& $ 95 $ \\
\hline
$ 36 $& $ 83 $& $ 92 $ \\
\hline
$ 37 $& $ 54 $& $ 77 $ \\
\hline
$ 38 $& $ 67 $& $ 36 $ \\
\hline
$ 39 $& $ 82 $& $ 90 $ \\
\hline
$ 40 $& $ 84 $& $ 116 $ \\
\hline
$ 41 $& $ 83 $& $ 116 $ \\
\hline
$ 42 $& $ 81 $& $ 77 $ \\
\hline
$ 43 $& $ 83 $& $ 96 $ \\
\hline
$ 44 $& $ 59 $& $ 110 $ \\
\hline
$ 45 $& $ 94 $& $ 103 $ \\
\hline
$ 46 $& $ 71 $& $ 102 $ \\
\hline
$ 47 $& $ 93 $& $ 78 $ \\
\hline
$ 48 $& $ 71 $& $ 103 $ \\
\hline
$ 49 $& $ 85 $& $ 77 $ \\
\hline
$ 50 $& $ 67 $& $ 105 $ \\
\hline
\end{tabular}
\begin{tabular}{|c||c|c|}
\multicolumn{3}{c}{$ 351  \leqslant i\leqslant 400 $}\\
\hline
{\bf $i'$}&{\bf $k$}&{\bf $l$}\\
\hline
\hline
$ 51 $& $ 68 $& $ 79 $ \\
\hline
$ 52 $& $ 95 $& $ 77 $ \\
\hline
$ 53 $& $ 35 $& $ 36 $ \\
\hline
$ 54 $& $ 90 $& $ 106 $ \\
\hline
$ 55 $& $ 36 $& $ 50 $ \\
\hline
$ 56 $& $ 79 $& $ 106 $ \\
\hline
$ 57 $& $ 62 $& $ 77 $ \\
\hline
$ 58 $& $ 40 $& $ 43 $ \\
\hline
$ 59 $& $ 74 $& $ 84 $ \\
\hline
$ 60 $& $ 75 $& $ 77 $ \\
\hline
$ 61 $& $ 57 $& $ 50 $ \\
\hline
$ 62 $& $ 101 $& $ 92 $ \\
\hline
$ 63 $& $ 44 $& $ 36 $ \\
\hline
$ 64 $& $ 102 $& $ 118 $ \\
\hline
$ 65 $& $ 97 $& $ 92 $ \\
\hline
$ 66 $& $ 62 $& $ 78 $ \\
\hline
$ 67 $& $ 78 $& $ 84 $ \\
\hline
$ 68 $& $ 65 $& $ 36 $ \\
\hline
$ 69 $& $ 31 $& $ 50 $ \\
\hline
$ 70 $& $ 92 $& $ 92 $ \\
\hline
$ 71 $& $ 83 $& $ 56 $ \\
\hline
$ 72 $& $ 94 $& $ 84 $ \\
\hline
$ 73 $& $ 95 $& $ 84 $ \\
\hline
$ 74 $& $ 87 $& $ 108 $ \\
\hline
$ 75 $& $ 98 $& $ 74 $ \\
\hline
$ 76 $& $ 99 $& $ 77 $ \\
\hline
$ 77 $& $ 90 $& $ 78 $ \\
\hline
$ 78 $& $ 81 $& $ 78 $ \\
\hline
$ 79 $& $ 97 $& $ 47 $ \\
\hline
$ 80 $& $ 58 $& $ 110 $ \\
\hline
$ 81 $& $ 78 $& $ 77 $ \\
\hline
$ 82 $& $ 54 $& $ 50 $ \\
\hline
$ 83 $& $ 73 $& $ 105 $ \\
\hline
$ 84 $& $ 72 $& $ 108 $ \\
\hline
$ 85 $& $ 88 $& $ 108 $ \\
\hline
$ 86 $& $ 86 $& $ 105 $ \\
\hline
$ 87 $& $ 103 $& $ 47 $ \\
\hline
$ 88 $& $ 105 $& $ 68 $ \\
\hline
$ 89 $& $ 59 $& $ 43 $ \\
\hline
$ 90 $& $ 65 $& $ 50 $ \\
\hline
$ 91 $& $ 90 $& $ 102 $ \\
\hline
$ 92 $& $ 91 $& $ 105 $ \\
\hline
$ 93 $& $ 98 $& $ 95 $ \\
\hline
$ 94 $& $ 49 $& $ 50 $ \\
\hline
$ 95 $& $ 52 $& $ 50 $ \\
\hline
$ 96 $& $ 85 $& $ 79 $ \\
\hline
$ 97 $& $ 73 $& $ 106 $ \\
\hline
$ 98 $& $ 92 $& $ 116 $ \\
\hline
$ 99 $& $ 101 $& $ 95 $ \\
\hline
$ 0 $& $ 102 $& $ 36 $ \\
\hline
\end{tabular}
\begin{tabular}{|c||c|c|}
\multicolumn{3}{c}{$ 401  \leqslant i\leqslant 450 $}\\
\hline
{\bf $i'$}&{\bf $k$}&{\bf $l$}\\
\hline
\hline
$ 1 $& $ 97 $& $ 116 $ \\
\hline
$ 2 $& $ 94 $& $ 79 $ \\
\hline
$ 3 $& $ 71 $& $ 84 $ \\
\hline
$ 4 $& $ 103 $& $ 105 $ \\
\hline
$ 5 $& $ 106 $& $ 74 $ \\
\hline
$ 6 $& $ 77 $& $ 79 $ \\
\hline
$ 7 $& $ 103 $& $ 77 $ \\
\hline
$ 8 $& $ 78 $& $ 50 $ \\
\hline
$ 9 $& $ 107 $& $ 118 $ \\
\hline
$ 10 $& $ 89 $& $ 105 $ \\
\hline
$ 11 $& $ 90 $& $ 79 $ \\
\hline
$ 12 $& $ 98 $& $ 106 $ \\
\hline
$ 13 $& $ 92 $& $ 106 $ \\
\hline
$ 14 $& $ 91 $& $ 36 $ \\
\hline
$ 15 $& $ 99 $& $ 106 $ \\
\hline
$ 16 $& $ 107 $& $ 92 $ \\
\hline
$ 17 $& $ 107 $& $ 110 $ \\
\hline
$ 18 $& $ 95 $& $ 79 $ \\
\hline
$ 19 $& $ 87 $& $ 43 $ \\
\hline
$ 20 $& $ 108 $& $ 118 $ \\
\hline
$ 21 $& $ 109 $& $ 8 $ \\
\hline
$ 22 $& $ 100 $& $ 77 $ \\
\hline
$ 23 $& $ 98 $& $ 79 $ \\
\hline
$ 24 $& $ 98 $& $ 102 $ \\
\hline
$ 25 $& $ 86 $& $ 50 $ \\
\hline
$ 26 $& $ 100 $& $ 102 $ \\
\hline
$ 27 $& $ 89 $& $ 50 $ \\
\hline
$ 28 $& $ 102 $& $ 79 $ \\
\hline
$ 29 $& $ 110 $& $ 77 $ \\
\hline
$ 30 $& $ 107 $& $ 87 $ \\
\hline
$ 31 $& $ 107 $& $ 79 $ \\
\hline
$ 32 $& $ 105 $& $ 102 $ \\
\hline
$ 33 $& $ 96 $& $ 105 $ \\
\hline
$ 34 $& $ 103 $& $ 36 $ \\
\hline
$ 35 $& $ 108 $& $ 108 $ \\
\hline
$ 36 $& $ 105 $& $ 105 $ \\
\hline
$ 37 $& $ 106 $& $ 36 $ \\
\hline
$ 38 $& $ 104 $& $ 43 $ \\
\hline
$ 39 $& $ 107 $& $ 102 $ \\
\hline
$ 40 $& $ 110 $& $ 50 $ \\
\hline
$ 41 $& $ 109 $& $ 102 $ \\
\hline
$ 42 $& $ 110 $& $ 106 $ \\
\hline
$ 43 $& $ 110 $& $ 105 $ \\
\hline
$ 44 $& $ 112 $& $ 8 $ \\
\hline
$ 45 $& $ 111 $& $ 8 $ \\
\hline
$ 46 $& $ 113 $& $ 8 $ \\
\hline
$ 47 $& $ 1 $& $ 74 $ \\
\hline
$ 48 $& $ 121 $& $ 8 $ \\
\hline
$ 49 $& $ 4 $& $ 116 $ \\
\hline
$ 50 $& $ 6 $& $ 87 $ \\
\hline
\end{tabular}
\begin{tabular}{|c||c|c|}
\multicolumn{3}{c}{$ 451  \leqslant i\leqslant 500 $}\\
\hline
{\bf $i'$}&{\bf $k$}&{\bf $l$}\\
\hline
\hline
$ 51 $& $ 1 $& $ 68 $ \\
\hline
$ 52 $& $ 1 $& $ 42 $ \\
\hline
$ 53 $& $ 114 $& $ 74 $ \\
\hline
$ 54 $& $ 4 $& $ 92 $ \\
\hline
$ 55 $& $ 111 $& $ 74 $ \\
\hline
$ 56 $& $ 122 $& $ 8 $ \\
\hline
$ 57 $& $ 1 $& $ 106 $ \\
\hline
$ 58 $& $ 1 $& $ 84 $ \\
\hline
$ 59 $& $ 116 $& $ 8 $ \\
\hline
$ 60 $& $ 118 $& $ 118 $ \\
\hline
$ 61 $& $ 3 $& $ 116 $ \\
\hline
$ 62 $& $ 117 $& $ 8 $ \\
\hline
$ 63 $& $ 116 $& $ 92 $ \\
\hline
$ 64 $& $ 119 $& $ 74 $ \\
\hline
$ 65 $& $ 1 $& $ 87 $ \\
\hline
$ 66 $& $ 112 $& $ 62 $ \\
\hline
$ 67 $& $ 115 $& $ 47 $ \\
\hline
$ 68 $& $ 16 $& $ 42 $ \\
\hline
$ 69 $& $ 130 $& $ 118 $ \\
\hline
$ 70 $& $ 2 $& $ 84 $ \\
\hline
$ 71 $& $ 10 $& $ 90 $ \\
\hline
$ 72 $& $ 122 $& $ 118 $ \\
\hline
$ 73 $& $ 4 $& $ 42 $ \\
\hline
$ 74 $& $ 13 $& $ 116 $ \\
\hline
$ 75 $& $ 121 $& $ 87 $ \\
\hline
$ 76 $& $ 5 $& $ 47 $ \\
\hline
$ 77 $& $ 129 $& $ 74 $ \\
\hline
$ 78 $& $ 6 $& $ 77 $ \\
\hline
$ 79 $& $ 4 $& $ 77 $ \\
\hline
$ 80 $& $ 9 $& $ 74 $ \\
\hline
$ 81 $& $ 115 $& $ 92 $ \\
\hline
$ 82 $& $ 111 $& $ 56 $ \\
\hline
$ 83 $& $ 122 $& $ 74 $ \\
\hline
$ 84 $& $ 122 $& $ 116 $ \\
\hline
$ 85 $& $ 122 $& $ 42 $ \\
\hline
$ 86 $& $ 134 $& $ 8 $ \\
\hline
$ 87 $& $ 7 $& $ 118 $ \\
\hline
$ 88 $& $ 137 $& $ 8 $ \\
\hline
$ 89 $& $ 133 $& $ 87 $ \\
\hline
$ 90 $& $ 12 $& $ 42 $ \\
\hline
$ 91 $& $ 118 $& $ 62 $ \\
\hline
$ 92 $& $ 123 $& $ 92 $ \\
\hline
$ 93 $& $ 12 $& $ 62 $ \\
\hline
$ 94 $& $ 117 $& $ 77 $ \\
\hline
$ 95 $& $ 122 $& $ 110 $ \\
\hline
$ 96 $& $ 119 $& $ 47 $ \\
\hline
$ 97 $& $ 122 $& $ 47 $ \\
\hline
$ 98 $& $ 116 $& $ 77 $ \\
\hline
$ 99 $& $ 113 $& $ 106 $ \\
\hline
$ 0 $& $ 113 $& $ 56 $ \\
\hline
\end{tabular}

}
\label{swc:8_2}
\end{table}


\begin{table}
\caption{Switching classes $\cS_{8,i}\ (501 \leqslant i\leqslant 750) $}
{\tiny
\begin{tabular}{|c||c|c|}
\multicolumn{3}{c}{$ 501  \leqslant i\leqslant 550 $}\\
\hline
{\bf $i'$}&{\bf $k$}&{\bf $l$}\\
\hline
\hline
$ 1 $& $ 116 $& $ 96 $ \\
\hline
$ 2 $& $ 119 $& $ 87 $ \\
\hline
$ 3 $& $ 115 $& $ 77 $ \\
\hline
$ 4 $& $ 140 $& $ 118 $ \\
\hline
$ 5 $& $ 119 $& $ 79 $ \\
\hline
$ 6 $& $ 122 $& $ 77 $ \\
\hline
$ 7 $& $ 119 $& $ 36 $ \\
\hline
$ 8 $& $ 111 $& $ 90 $ \\
\hline
$ 9 $& $ 6 $& $ 117 $ \\
\hline
$ 10 $& $ 4 $& $ 43 $ \\
\hline
$ 11 $& $ 111 $& $ 103 $ \\
\hline
$ 12 $& $ 127 $& $ 8 $ \\
\hline
$ 13 $& $ 115 $& $ 96 $ \\
\hline
$ 14 $& $ 145 $& $ 106 $ \\
\hline
$ 15 $& $ 120 $& $ 116 $ \\
\hline
$ 16 $& $ 130 $& $ 8 $ \\
\hline
$ 17 $& $ 122 $& $ 36 $ \\
\hline
$ 18 $& $ 12 $& $ 116 $ \\
\hline
$ 19 $& $ 128 $& $ 8 $ \\
\hline
$ 20 $& $ 142 $& $ 74 $ \\
\hline
$ 21 $& $ 4 $& $ 84 $ \\
\hline
$ 22 $& $ 121 $& $ 105 $ \\
\hline
$ 23 $& $ 111 $& $ 87 $ \\
\hline
$ 24 $& $ 113 $& $ 116 $ \\
\hline
$ 25 $& $ 127 $& $ 74 $ \\
\hline
$ 26 $& $ 8 $& $ 103 $ \\
\hline
$ 27 $& $ 123 $& $ 87 $ \\
\hline
$ 28 $& $ 4 $& $ 106 $ \\
\hline
$ 29 $& $ 168 $& $ 8 $ \\
\hline
$ 30 $& $ 122 $& $ 43 $ \\
\hline
$ 31 $& $ 18 $& $ 62 $ \\
\hline
$ 32 $& $ 136 $& $ 74 $ \\
\hline
$ 33 $& $ 123 $& $ 96 $ \\
\hline
$ 34 $& $ 16 $& $ 79 $ \\
\hline
$ 35 $& $ 157 $& $ 74 $ \\
\hline
$ 36 $& $ 115 $& $ 42 $ \\
\hline
$ 37 $& $ 25 $& $ 68 $ \\
\hline
$ 38 $& $ 29 $& $ 117 $ \\
\hline
$ 39 $& $ 149 $& $ 87 $ \\
\hline
$ 40 $& $ 117 $& $ 84 $ \\
\hline
$ 41 $& $ 13 $& $ 56 $ \\
\hline
$ 42 $& $ 125 $& $ 78 $ \\
\hline
$ 43 $& $ 131 $& $ 116 $ \\
\hline
$ 44 $& $ 137 $& $ 92 $ \\
\hline
$ 45 $& $ 123 $& $ 116 $ \\
\hline
$ 46 $& $ 5 $& $ 84 $ \\
\hline
$ 47 $& $ 148 $& $ 8 $ \\
\hline
$ 48 $& $ 123 $& $ 106 $ \\
\hline
$ 49 $& $ 123 $& $ 77 $ \\
\hline
$ 50 $& $ 6 $& $ 110 $ \\
\hline
\end{tabular}
\begin{tabular}{|c||c|c|}
\multicolumn{3}{c}{$ 551  \leqslant i\leqslant 600 $}\\
\hline
{\bf $i'$}&{\bf $k$}&{\bf $l$}\\
\hline
\hline
$ 51 $& $ 119 $& $ 90 $ \\
\hline
$ 52 $& $ 143 $& $ 118 $ \\
\hline
$ 53 $& $ 119 $& $ 102 $ \\
\hline
$ 54 $& $ 132 $& $ 103 $ \\
\hline
$ 55 $& $ 120 $& $ 92 $ \\
\hline
$ 56 $& $ 11 $& $ 47 $ \\
\hline
$ 57 $& $ 131 $& $ 74 $ \\
\hline
$ 58 $& $ 122 $& $ 92 $ \\
\hline
$ 59 $& $ 132 $& $ 116 $ \\
\hline
$ 60 $& $ 160 $& $ 8 $ \\
\hline
$ 61 $& $ 121 $& $ 68 $ \\
\hline
$ 62 $& $ 144 $& $ 87 $ \\
\hline
$ 63 $& $ 117 $& $ 96 $ \\
\hline
$ 64 $& $ 140 $& $ 47 $ \\
\hline
$ 65 $& $ 164 $& $ 118 $ \\
\hline
$ 66 $& $ 127 $& $ 116 $ \\
\hline
$ 67 $& $ 12 $& $ 74 $ \\
\hline
$ 68 $& $ 160 $& $ 117 $ \\
\hline
$ 69 $& $ 9 $& $ 84 $ \\
\hline
$ 70 $& $ 122 $& $ 79 $ \\
\hline
$ 71 $& $ 114 $& $ 77 $ \\
\hline
$ 72 $& $ 134 $& $ 96 $ \\
\hline
$ 73 $& $ 127 $& $ 87 $ \\
\hline
$ 74 $& $ 135 $& $ 84 $ \\
\hline
$ 75 $& $ 15 $& $ 106 $ \\
\hline
$ 76 $& $ 124 $& $ 106 $ \\
\hline
$ 77 $& $ 132 $& $ 56 $ \\
\hline
$ 78 $& $ 162 $& $ 117 $ \\
\hline
$ 79 $& $ 154 $& $ 8 $ \\
\hline
$ 80 $& $ 16 $& $ 43 $ \\
\hline
$ 81 $& $ 150 $& $ 42 $ \\
\hline
$ 82 $& $ 9 $& $ 56 $ \\
\hline
$ 83 $& $ 6 $& $ 106 $ \\
\hline
$ 84 $& $ 120 $& $ 79 $ \\
\hline
$ 85 $& $ 168 $& $ 87 $ \\
\hline
$ 86 $& $ 116 $& $ 79 $ \\
\hline
$ 87 $& $ 111 $& $ 79 $ \\
\hline
$ 88 $& $ 145 $& $ 74 $ \\
\hline
$ 89 $& $ 124 $& $ 56 $ \\
\hline
$ 90 $& $ 129 $& $ 92 $ \\
\hline
$ 91 $& $ 122 $& $ 50 $ \\
\hline
$ 92 $& $ 132 $& $ 77 $ \\
\hline
$ 93 $& $ 149 $& $ 116 $ \\
\hline
$ 94 $& $ 122 $& $ 117 $ \\
\hline
$ 95 $& $ 153 $& $ 8 $ \\
\hline
$ 96 $& $ 21 $& $ 42 $ \\
\hline
$ 97 $& $ 142 $& $ 77 $ \\
\hline
$ 98 $& $ 15 $& $ 77 $ \\
\hline
$ 99 $& $ 148 $& $ 116 $ \\
\hline
$ 0 $& $ 22 $& $ 96 $ \\
\hline
\end{tabular}
\begin{tabular}{|c||c|c|}
\multicolumn{3}{c}{$ 601  \leqslant i\leqslant 650 $}\\
\hline
{\bf $i'$}&{\bf $k$}&{\bf $l$}\\
\hline
\hline
$ 1 $& $ 138 $& $ 96 $ \\
\hline
$ 2 $& $ 124 $& $ 95 $ \\
\hline
$ 3 $& $ 16 $& $ 96 $ \\
\hline
$ 4 $& $ 137 $& $ 68 $ \\
\hline
$ 5 $& $ 163 $& $ 87 $ \\
\hline
$ 6 $& $ 14 $& $ 56 $ \\
\hline
$ 7 $& $ 151 $& $ 8 $ \\
\hline
$ 8 $& $ 113 $& $ 79 $ \\
\hline
$ 9 $& $ 131 $& $ 118 $ \\
\hline
$ 10 $& $ 168 $& $ 47 $ \\
\hline
$ 11 $& $ 130 $& $ 68 $ \\
\hline
$ 12 $& $ 161 $& $ 8 $ \\
\hline
$ 13 $& $ 176 $& $ 96 $ \\
\hline
$ 14 $& $ 114 $& $ 102 $ \\
\hline
$ 15 $& $ 143 $& $ 116 $ \\
\hline
$ 16 $& $ 164 $& $ 77 $ \\
\hline
$ 17 $& $ 124 $& $ 74 $ \\
\hline
$ 18 $& $ 114 $& $ 84 $ \\
\hline
$ 19 $& $ 168 $& $ 42 $ \\
\hline
$ 20 $& $ 13 $& $ 47 $ \\
\hline
$ 21 $& $ 24 $& $ 50 $ \\
\hline
$ 22 $& $ 131 $& $ 68 $ \\
\hline
$ 23 $& $ 120 $& $ 103 $ \\
\hline
$ 24 $& $ 148 $& $ 84 $ \\
\hline
$ 25 $& $ 153 $& $ 118 $ \\
\hline
$ 26 $& $ 191 $& $ 74 $ \\
\hline
$ 27 $& $ 160 $& $ 118 $ \\
\hline
$ 28 $& $ 136 $& $ 117 $ \\
\hline
$ 29 $& $ 136 $& $ 36 $ \\
\hline
$ 30 $& $ 124 $& $ 103 $ \\
\hline
$ 31 $& $ 39 $& $ 118 $ \\
\hline
$ 32 $& $ 122 $& $ 90 $ \\
\hline
$ 33 $& $ 180 $& $ 74 $ \\
\hline
$ 34 $& $ 189 $& $ 118 $ \\
\hline
$ 35 $& $ 165 $& $ 92 $ \\
\hline
$ 36 $& $ 16 $& $ 92 $ \\
\hline
$ 37 $& $ 141 $& $ 62 $ \\
\hline
$ 38 $& $ 124 $& $ 92 $ \\
\hline
$ 39 $& $ 172 $& $ 68 $ \\
\hline
$ 40 $& $ 117 $& $ 102 $ \\
\hline
$ 41 $& $ 126 $& $ 87 $ \\
\hline
$ 42 $& $ 165 $& $ 68 $ \\
\hline
$ 43 $& $ 144 $& $ 106 $ \\
\hline
$ 44 $& $ 132 $& $ 79 $ \\
\hline
$ 45 $& $ 169 $& $ 117 $ \\
\hline
$ 46 $& $ 125 $& $ 77 $ \\
\hline
$ 47 $& $ 157 $& $ 79 $ \\
\hline
$ 48 $& $ 156 $& $ 106 $ \\
\hline
$ 49 $& $ 45 $& $ 118 $ \\
\hline
$ 50 $& $ 21 $& $ 116 $ \\
\hline
\end{tabular}
\begin{tabular}{|c||c|c|}
\multicolumn{3}{c}{$ 651  \leqslant i\leqslant 700 $}\\
\hline
{\bf $i'$}&{\bf $k$}&{\bf $l$}\\
\hline
\hline
$ 51 $& $ 132 $& $ 106 $ \\
\hline
$ 52 $& $ 122 $& $ 108 $ \\
\hline
$ 53 $& $ 133 $& $ 92 $ \\
\hline
$ 54 $& $ 14 $& $ 110 $ \\
\hline
$ 55 $& $ 148 $& $ 106 $ \\
\hline
$ 56 $& $ 173 $& $ 116 $ \\
\hline
$ 57 $& $ 30 $& $ 95 $ \\
\hline
$ 58 $& $ 181 $& $ 74 $ \\
\hline
$ 59 $& $ 159 $& $ 96 $ \\
\hline
$ 60 $& $ 45 $& $ 110 $ \\
\hline
$ 61 $& $ 124 $& $ 62 $ \\
\hline
$ 62 $& $ 188 $& $ 74 $ \\
\hline
$ 63 $& $ 14 $& $ 62 $ \\
\hline
$ 64 $& $ 144 $& $ 77 $ \\
\hline
$ 65 $& $ 36 $& $ 96 $ \\
\hline
$ 66 $& $ 133 $& $ 50 $ \\
\hline
$ 67 $& $ 164 $& $ 110 $ \\
\hline
$ 68 $& $ 158 $& $ 116 $ \\
\hline
$ 69 $& $ 114 $& $ 78 $ \\
\hline
$ 70 $& $ 112 $& $ 110 $ \\
\hline
$ 71 $& $ 139 $& $ 62 $ \\
\hline
$ 72 $& $ 28 $& $ 90 $ \\
\hline
$ 73 $& $ 142 $& $ 106 $ \\
\hline
$ 74 $& $ 123 $& $ 43 $ \\
\hline
$ 75 $& $ 167 $& $ 47 $ \\
\hline
$ 76 $& $ 10 $& $ 43 $ \\
\hline
$ 77 $& $ 14 $& $ 43 $ \\
\hline
$ 78 $& $ 172 $& $ 8 $ \\
\hline
$ 79 $& $ 28 $& $ 96 $ \\
\hline
$ 80 $& $ 13 $& $ 102 $ \\
\hline
$ 81 $& $ 140 $& $ 92 $ \\
\hline
$ 82 $& $ 131 $& $ 90 $ \\
\hline
$ 83 $& $ 137 $& $ 79 $ \\
\hline
$ 84 $& $ 139 $& $ 90 $ \\
\hline
$ 85 $& $ 127 $& $ 78 $ \\
\hline
$ 86 $& $ 125 $& $ 90 $ \\
\hline
$ 87 $& $ 114 $& $ 56 $ \\
\hline
$ 88 $& $ 132 $& $ 96 $ \\
\hline
$ 89 $& $ 115 $& $ 95 $ \\
\hline
$ 90 $& $ 178 $& $ 8 $ \\
\hline
$ 91 $& $ 45 $& $ 96 $ \\
\hline
$ 92 $& $ 156 $& $ 117 $ \\
\hline
$ 93 $& $ 23 $& $ 96 $ \\
\hline
$ 94 $& $ 129 $& $ 117 $ \\
\hline
$ 95 $& $ 134 $& $ 118 $ \\
\hline
$ 96 $& $ 135 $& $ 42 $ \\
\hline
$ 97 $& $ 134 $& $ 77 $ \\
\hline
$ 98 $& $ 156 $& $ 74 $ \\
\hline
$ 99 $& $ 154 $& $ 74 $ \\
\hline
$ 0 $& $ 137 $& $ 117 $ \\
\hline
\end{tabular}
\begin{tabular}{|c||c|c|}
\multicolumn{3}{c}{$ 701  \leqslant i\leqslant 750 $}\\
\hline
{\bf $i'$}&{\bf $k$}&{\bf $l$}\\
\hline
\hline
$ 1 $& $ 154 $& $ 96 $ \\
\hline
$ 2 $& $ 139 $& $ 108 $ \\
\hline
$ 3 $& $ 151 $& $ 96 $ \\
\hline
$ 4 $& $ 37 $& $ 42 $ \\
\hline
$ 5 $& $ 123 $& $ 79 $ \\
\hline
$ 6 $& $ 125 $& $ 117 $ \\
\hline
$ 7 $& $ 182 $& $ 78 $ \\
\hline
$ 8 $& $ 121 $& $ 108 $ \\
\hline
$ 9 $& $ 161 $& $ 118 $ \\
\hline
$ 10 $& $ 192 $& $ 74 $ \\
\hline
$ 11 $& $ 37 $& $ 106 $ \\
\hline
$ 12 $& $ 164 $& $ 87 $ \\
\hline
$ 13 $& $ 155 $& $ 117 $ \\
\hline
$ 14 $& $ 114 $& $ 50 $ \\
\hline
$ 15 $& $ 42 $& $ 77 $ \\
\hline
$ 16 $& $ 115 $& $ 43 $ \\
\hline
$ 17 $& $ 132 $& $ 102 $ \\
\hline
$ 18 $& $ 138 $& $ 105 $ \\
\hline
$ 19 $& $ 184 $& $ 8 $ \\
\hline
$ 20 $& $ 20 $& $ 68 $ \\
\hline
$ 21 $& $ 119 $& $ 117 $ \\
\hline
$ 22 $& $ 157 $& $ 56 $ \\
\hline
$ 23 $& $ 146 $& $ 77 $ \\
\hline
$ 24 $& $ 123 $& $ 84 $ \\
\hline
$ 25 $& $ 41 $& $ 118 $ \\
\hline
$ 26 $& $ 122 $& $ 84 $ \\
\hline
$ 27 $& $ 125 $& $ 84 $ \\
\hline
$ 28 $& $ 114 $& $ 62 $ \\
\hline
$ 29 $& $ 192 $& $ 117 $ \\
\hline
$ 30 $& $ 44 $& $ 116 $ \\
\hline
$ 31 $& $ 141 $& $ 79 $ \\
\hline
$ 32 $& $ 158 $& $ 74 $ \\
\hline
$ 33 $& $ 157 $& $ 87 $ \\
\hline
$ 34 $& $ 147 $& $ 116 $ \\
\hline
$ 35 $& $ 27 $& $ 87 $ \\
\hline
$ 36 $& $ 159 $& $ 92 $ \\
\hline
$ 37 $& $ 159 $& $ 79 $ \\
\hline
$ 38 $& $ 174 $& $ 92 $ \\
\hline
$ 39 $& $ 189 $& $ 116 $ \\
\hline
$ 40 $& $ 28 $& $ 103 $ \\
\hline
$ 41 $& $ 153 $& $ 92 $ \\
\hline
$ 42 $& $ 26 $& $ 110 $ \\
\hline
$ 43 $& $ 46 $& $ 116 $ \\
\hline
$ 44 $& $ 153 $& $ 68 $ \\
\hline
$ 45 $& $ 129 $& $ 50 $ \\
\hline
$ 46 $& $ 129 $& $ 68 $ \\
\hline
$ 47 $& $ 143 $& $ 42 $ \\
\hline
$ 48 $& $ 113 $& $ 36 $ \\
\hline
$ 49 $& $ 19 $& $ 102 $ \\
\hline
$ 50 $& $ 143 $& $ 84 $ \\
\hline
\end{tabular}

}
\label{swc:8_3}
\end{table}


\begin{table}
\caption{Switching classes $\cS_{8,i}\ (751\leqslant i\leqslant 1000) $}
{\tiny
\begin{tabular}{|c||c|c|}
\multicolumn{3}{c}{$ 751  \leqslant i\leqslant 800 $}\\
\hline
{\bf $i'$}&{\bf $k$}&{\bf $l$}\\
\hline
\hline
$ 51 $& $ 131 $& $ 117 $ \\
\hline
$ 52 $& $ 154 $& $ 62 $ \\
\hline
$ 53 $& $ 187 $& $ 8 $ \\
\hline
$ 54 $& $ 138 $& $ 74 $ \\
\hline
$ 55 $& $ 166 $& $ 47 $ \\
\hline
$ 56 $& $ 159 $& $ 74 $ \\
\hline
$ 57 $& $ 122 $& $ 103 $ \\
\hline
$ 58 $& $ 7 $& $ 47 $ \\
\hline
$ 59 $& $ 161 $& $ 74 $ \\
\hline
$ 60 $& $ 124 $& $ 47 $ \\
\hline
$ 61 $& $ 136 $& $ 43 $ \\
\hline
$ 62 $& $ 187 $& $ 47 $ \\
\hline
$ 63 $& $ 41 $& $ 42 $ \\
\hline
$ 64 $& $ 25 $& $ 110 $ \\
\hline
$ 65 $& $ 45 $& $ 90 $ \\
\hline
$ 66 $& $ 186 $& $ 118 $ \\
\hline
$ 67 $& $ 155 $& $ 68 $ \\
\hline
$ 68 $& $ 135 $& $ 68 $ \\
\hline
$ 69 $& $ 147 $& $ 50 $ \\
\hline
$ 70 $& $ 61 $& $ 56 $ \\
\hline
$ 71 $& $ 150 $& $ 50 $ \\
\hline
$ 72 $& $ 41 $& $ 56 $ \\
\hline
$ 73 $& $ 209 $& $ 8 $ \\
\hline
$ 74 $& $ 46 $& $ 95 $ \\
\hline
$ 75 $& $ 53 $& $ 62 $ \\
\hline
$ 76 $& $ 131 $& $ 62 $ \\
\hline
$ 77 $& $ 124 $& $ 77 $ \\
\hline
$ 78 $& $ 172 $& $ 118 $ \\
\hline
$ 79 $& $ 185 $& $ 84 $ \\
\hline
$ 80 $& $ 173 $& $ 68 $ \\
\hline
$ 81 $& $ 123 $& $ 78 $ \\
\hline
$ 82 $& $ 140 $& $ 95 $ \\
\hline
$ 83 $& $ 120 $& $ 95 $ \\
\hline
$ 84 $& $ 210 $& $ 8 $ \\
\hline
$ 85 $& $ 161 $& $ 62 $ \\
\hline
$ 86 $& $ 205 $& $ 8 $ \\
\hline
$ 87 $& $ 144 $& $ 78 $ \\
\hline
$ 88 $& $ 161 $& $ 77 $ \\
\hline
$ 89 $& $ 23 $& $ 62 $ \\
\hline
$ 90 $& $ 21 $& $ 110 $ \\
\hline
$ 91 $& $ 156 $& $ 95 $ \\
\hline
$ 92 $& $ 150 $& $ 84 $ \\
\hline
$ 93 $& $ 183 $& $ 92 $ \\
\hline
$ 94 $& $ 40 $& $ 106 $ \\
\hline
$ 95 $& $ 39 $& $ 56 $ \\
\hline
$ 96 $& $ 193 $& $ 47 $ \\
\hline
$ 97 $& $ 206 $& $ 118 $ \\
\hline
$ 98 $& $ 178 $& $ 68 $ \\
\hline
$ 99 $& $ 207 $& $ 116 $ \\
\hline
$ 0 $& $ 148 $& $ 78 $ \\
\hline
\end{tabular}
\begin{tabular}{|c||c|c|}
\multicolumn{3}{c}{$ 801  \leqslant i\leqslant 850 $}\\
\hline
{\bf $i'$}&{\bf $k$}&{\bf $l$}\\
\hline
\hline
$ 1 $& $ 185 $& $ 79 $ \\
\hline
$ 2 $& $ 131 $& $ 78 $ \\
\hline
$ 3 $& $ 203 $& $ 92 $ \\
\hline
$ 4 $& $ 154 $& $ 68 $ \\
\hline
$ 5 $& $ 119 $& $ 105 $ \\
\hline
$ 6 $& $ 162 $& $ 77 $ \\
\hline
$ 7 $& $ 188 $& $ 118 $ \\
\hline
$ 8 $& $ 194 $& $ 77 $ \\
\hline
$ 9 $& $ 163 $& $ 116 $ \\
\hline
$ 10 $& $ 202 $& $ 103 $ \\
\hline
$ 11 $& $ 207 $& $ 118 $ \\
\hline
$ 12 $& $ 123 $& $ 56 $ \\
\hline
$ 13 $& $ 200 $& $ 8 $ \\
\hline
$ 14 $& $ 197 $& $ 118 $ \\
\hline
$ 15 $& $ 157 $& $ 90 $ \\
\hline
$ 16 $& $ 188 $& $ 62 $ \\
\hline
$ 17 $& $ 180 $& $ 47 $ \\
\hline
$ 18 $& $ 196 $& $ 118 $ \\
\hline
$ 19 $& $ 17 $& $ 103 $ \\
\hline
$ 20 $& $ 156 $& $ 42 $ \\
\hline
$ 21 $& $ 158 $& $ 68 $ \\
\hline
$ 22 $& $ 179 $& $ 84 $ \\
\hline
$ 23 $& $ 161 $& $ 79 $ \\
\hline
$ 24 $& $ 193 $& $ 106 $ \\
\hline
$ 25 $& $ 187 $& $ 74 $ \\
\hline
$ 26 $& $ 127 $& $ 42 $ \\
\hline
$ 27 $& $ 54 $& $ 62 $ \\
\hline
$ 28 $& $ 156 $& $ 90 $ \\
\hline
$ 29 $& $ 190 $& $ 87 $ \\
\hline
$ 30 $& $ 42 $& $ 90 $ \\
\hline
$ 31 $& $ 146 $& $ 84 $ \\
\hline
$ 32 $& $ 150 $& $ 79 $ \\
\hline
$ 33 $& $ 147 $& $ 56 $ \\
\hline
$ 34 $& $ 139 $& $ 36 $ \\
\hline
$ 35 $& $ 143 $& $ 95 $ \\
\hline
$ 36 $& $ 159 $& $ 43 $ \\
\hline
$ 37 $& $ 60 $& $ 62 $ \\
\hline
$ 38 $& $ 163 $& $ 47 $ \\
\hline
$ 39 $& $ 139 $& $ 105 $ \\
\hline
$ 40 $& $ 164 $& $ 36 $ \\
\hline
$ 41 $& $ 143 $& $ 103 $ \\
\hline
$ 42 $& $ 183 $& $ 116 $ \\
\hline
$ 43 $& $ 158 $& $ 102 $ \\
\hline
$ 44 $& $ 154 $& $ 102 $ \\
\hline
$ 45 $& $ 189 $& $ 117 $ \\
\hline
$ 46 $& $ 145 $& $ 77 $ \\
\hline
$ 47 $& $ 194 $& $ 106 $ \\
\hline
$ 48 $& $ 208 $& $ 74 $ \\
\hline
$ 49 $& $ 139 $& $ 50 $ \\
\hline
$ 50 $& $ 154 $& $ 90 $ \\
\hline
\end{tabular}
\begin{tabular}{|c||c|c|}
\multicolumn{3}{c}{$ 851  \leqslant i\leqslant 900 $}\\
\hline
{\bf $i'$}&{\bf $k$}&{\bf $l$}\\
\hline
\hline
$ 51 $& $ 156 $& $ 62 $ \\
\hline
$ 52 $& $ 147 $& $ 77 $ \\
\hline
$ 53 $& $ 183 $& $ 103 $ \\
\hline
$ 54 $& $ 162 $& $ 108 $ \\
\hline
$ 55 $& $ 155 $& $ 79 $ \\
\hline
$ 56 $& $ 143 $& $ 102 $ \\
\hline
$ 57 $& $ 44 $& $ 42 $ \\
\hline
$ 58 $& $ 134 $& $ 95 $ \\
\hline
$ 59 $& $ 122 $& $ 78 $ \\
\hline
$ 60 $& $ 209 $& $ 74 $ \\
\hline
$ 61 $& $ 124 $& $ 43 $ \\
\hline
$ 62 $& $ 116 $& $ 68 $ \\
\hline
$ 63 $& $ 161 $& $ 68 $ \\
\hline
$ 64 $& $ 124 $& $ 79 $ \\
\hline
$ 65 $& $ 172 $& $ 62 $ \\
\hline
$ 66 $& $ 177 $& $ 74 $ \\
\hline
$ 67 $& $ 172 $& $ 74 $ \\
\hline
$ 68 $& $ 165 $& $ 84 $ \\
\hline
$ 69 $& $ 186 $& $ 79 $ \\
\hline
$ 70 $& $ 177 $& $ 62 $ \\
\hline
$ 71 $& $ 136 $& $ 90 $ \\
\hline
$ 72 $& $ 189 $& $ 92 $ \\
\hline
$ 73 $& $ 158 $& $ 95 $ \\
\hline
$ 74 $& $ 34 $& $ 36 $ \\
\hline
$ 75 $& $ 177 $& $ 79 $ \\
\hline
$ 76 $& $ 144 $& $ 108 $ \\
\hline
$ 77 $& $ 178 $& $ 78 $ \\
\hline
$ 78 $& $ 117 $& $ 68 $ \\
\hline
$ 79 $& $ 155 $& $ 78 $ \\
\hline
$ 80 $& $ 161 $& $ 84 $ \\
\hline
$ 81 $& $ 160 $& $ 68 $ \\
\hline
$ 82 $& $ 50 $& $ 84 $ \\
\hline
$ 83 $& $ 141 $& $ 108 $ \\
\hline
$ 84 $& $ 130 $& $ 36 $ \\
\hline
$ 85 $& $ 167 $& $ 84 $ \\
\hline
$ 86 $& $ 122 $& $ 95 $ \\
\hline
$ 87 $& $ 43 $& $ 56 $ \\
\hline
$ 88 $& $ 189 $& $ 87 $ \\
\hline
$ 89 $& $ 150 $& $ 77 $ \\
\hline
$ 90 $& $ 169 $& $ 62 $ \\
\hline
$ 91 $& $ 183 $& $ 102 $ \\
\hline
$ 92 $& $ 183 $& $ 110 $ \\
\hline
$ 93 $& $ 199 $& $ 90 $ \\
\hline
$ 94 $& $ 151 $& $ 110 $ \\
\hline
$ 95 $& $ 206 $& $ 92 $ \\
\hline
$ 96 $& $ 172 $& $ 105 $ \\
\hline
$ 97 $& $ 201 $& $ 62 $ \\
\hline
$ 98 $& $ 164 $& $ 106 $ \\
\hline
$ 99 $& $ 204 $& $ 56 $ \\
\hline
$ 0 $& $ 219 $& $ 8 $ \\
\hline
\end{tabular}
\begin{tabular}{|c||c|c|}
\multicolumn{3}{c}{$ 901  \leqslant i\leqslant 950 $}\\
\hline
{\bf $i'$}&{\bf $k$}&{\bf $l$}\\
\hline
\hline
$ 1 $& $ 33 $& $ 42 $ \\
\hline
$ 2 $& $ 189 $& $ 106 $ \\
\hline
$ 3 $& $ 175 $& $ 92 $ \\
\hline
$ 4 $& $ 209 $& $ 92 $ \\
\hline
$ 5 $& $ 64 $& $ 42 $ \\
\hline
$ 6 $& $ 212 $& $ 78 $ \\
\hline
$ 7 $& $ 166 $& $ 68 $ \\
\hline
$ 8 $& $ 36 $& $ 103 $ \\
\hline
$ 9 $& $ 169 $& $ 95 $ \\
\hline
$ 10 $& $ 212 $& $ 42 $ \\
\hline
$ 11 $& $ 135 $& $ 50 $ \\
\hline
$ 12 $& $ 183 $& $ 68 $ \\
\hline
$ 13 $& $ 222 $& $ 118 $ \\
\hline
$ 14 $& $ 178 $& $ 90 $ \\
\hline
$ 15 $& $ 199 $& $ 68 $ \\
\hline
$ 16 $& $ 62 $& $ 96 $ \\
\hline
$ 17 $& $ 216 $& $ 103 $ \\
\hline
$ 18 $& $ 62 $& $ 103 $ \\
\hline
$ 19 $& $ 191 $& $ 50 $ \\
\hline
$ 20 $& $ 203 $& $ 118 $ \\
\hline
$ 21 $& $ 209 $& $ 84 $ \\
\hline
$ 22 $& $ 194 $& $ 62 $ \\
\hline
$ 23 $& $ 211 $& $ 90 $ \\
\hline
$ 24 $& $ 163 $& $ 84 $ \\
\hline
$ 25 $& $ 218 $& $ 8 $ \\
\hline
$ 26 $& $ 176 $& $ 90 $ \\
\hline
$ 27 $& $ 180 $& $ 106 $ \\
\hline
$ 28 $& $ 147 $& $ 84 $ \\
\hline
$ 29 $& $ 192 $& $ 105 $ \\
\hline
$ 30 $& $ 174 $& $ 47 $ \\
\hline
$ 31 $& $ 157 $& $ 95 $ \\
\hline
$ 32 $& $ 223 $& $ 118 $ \\
\hline
$ 33 $& $ 156 $& $ 79 $ \\
\hline
$ 34 $& $ 204 $& $ 92 $ \\
\hline
$ 35 $& $ 188 $& $ 84 $ \\
\hline
$ 36 $& $ 37 $& $ 62 $ \\
\hline
$ 37 $& $ 207 $& $ 79 $ \\
\hline
$ 38 $& $ 154 $& $ 43 $ \\
\hline
$ 39 $& $ 175 $& $ 42 $ \\
\hline
$ 40 $& $ 196 $& $ 96 $ \\
\hline
$ 41 $& $ 216 $& $ 118 $ \\
\hline
$ 42 $& $ 169 $& $ 103 $ \\
\hline
$ 43 $& $ 187 $& $ 106 $ \\
\hline
$ 44 $& $ 66 $& $ 42 $ \\
\hline
$ 45 $& $ 171 $& $ 117 $ \\
\hline
$ 46 $& $ 213 $& $ 8 $ \\
\hline
$ 47 $& $ 147 $& $ 117 $ \\
\hline
$ 48 $& $ 191 $& $ 43 $ \\
\hline
$ 49 $& $ 131 $& $ 103 $ \\
\hline
$ 50 $& $ 60 $& $ 90 $ \\
\hline
\end{tabular}
\begin{tabular}{|c||c|c|}
\multicolumn{3}{c}{$ 951  \leqslant i\leqslant 1000 $}\\
\hline
{\bf $i'$}&{\bf $k$}&{\bf $l$}\\
\hline
\hline
$ 51 $& $ 123 $& $ 110 $ \\
\hline
$ 52 $& $ 185 $& $ 77 $ \\
\hline
$ 53 $& $ 40 $& $ 42 $ \\
\hline
$ 54 $& $ 174 $& $ 87 $ \\
\hline
$ 55 $& $ 155 $& $ 96 $ \\
\hline
$ 56 $& $ 156 $& $ 36 $ \\
\hline
$ 57 $& $ 130 $& $ 105 $ \\
\hline
$ 58 $& $ 152 $& $ 84 $ \\
\hline
$ 59 $& $ 145 $& $ 36 $ \\
\hline
$ 60 $& $ 177 $& $ 90 $ \\
\hline
$ 61 $& $ 152 $& $ 50 $ \\
\hline
$ 62 $& $ 52 $& $ 103 $ \\
\hline
$ 63 $& $ 153 $& $ 42 $ \\
\hline
$ 64 $& $ 156 $& $ 84 $ \\
\hline
$ 65 $& $ 21 $& $ 43 $ \\
\hline
$ 66 $& $ 179 $& $ 106 $ \\
\hline
$ 67 $& $ 31 $& $ 90 $ \\
\hline
$ 68 $& $ 50 $& $ 103 $ \\
\hline
$ 69 $& $ 161 $& $ 108 $ \\
\hline
$ 70 $& $ 155 $& $ 77 $ \\
\hline
$ 71 $& $ 174 $& $ 102 $ \\
\hline
$ 72 $& $ 164 $& $ 108 $ \\
\hline
$ 73 $& $ 200 $& $ 95 $ \\
\hline
$ 74 $& $ 194 $& $ 84 $ \\
\hline
$ 75 $& $ 160 $& $ 42 $ \\
\hline
$ 76 $& $ 169 $& $ 84 $ \\
\hline
$ 77 $& $ 167 $& $ 110 $ \\
\hline
$ 78 $& $ 190 $& $ 84 $ \\
\hline
$ 79 $& $ 173 $& $ 108 $ \\
\hline
$ 80 $& $ 44 $& $ 110 $ \\
\hline
$ 81 $& $ 40 $& $ 50 $ \\
\hline
$ 82 $& $ 217 $& $ 8 $ \\
\hline
$ 83 $& $ 201 $& $ 116 $ \\
\hline
$ 84 $& $ 221 $& $ 118 $ \\
\hline
$ 85 $& $ 76 $& $ 96 $ \\
\hline
$ 86 $& $ 197 $& $ 77 $ \\
\hline
$ 87 $& $ 148 $& $ 108 $ \\
\hline
$ 88 $& $ 157 $& $ 62 $ \\
\hline
$ 89 $& $ 180 $& $ 77 $ \\
\hline
$ 90 $& $ 189 $& $ 96 $ \\
\hline
$ 91 $& $ 181 $& $ 79 $ \\
\hline
$ 92 $& $ 186 $& $ 102 $ \\
\hline
$ 93 $& $ 158 $& $ 47 $ \\
\hline
$ 94 $& $ 71 $& $ 68 $ \\
\hline
$ 95 $& $ 182 $& $ 87 $ \\
\hline
$ 96 $& $ 128 $& $ 43 $ \\
\hline
$ 97 $& $ 159 $& $ 50 $ \\
\hline
$ 98 $& $ 159 $& $ 84 $ \\
\hline
$ 99 $& $ 183 $& $ 90 $ \\
\hline
$ 0 $& $ 199 $& $ 118 $ \\
\hline
\end{tabular}

}
\label{swc:8_4}
\end{table}


\begin{table}
\caption{Switching classes $\cS_{8,1000+i}\ (1 \leqslant i\leqslant 242) $}
{\tiny
\begin{tabular}{|c||c|c|}
\multicolumn{3}{c}{$ 1  \leqslant i\leqslant 50 $}\\
\hline
{\bf $i'$}&{\bf $k$}&{\bf $l$}\\
\hline
\hline
$ 1 $& $ 175 $& $ 84 $ \\
\hline
$ 2 $& $ 30 $& $ 102 $ \\
\hline
$ 3 $& $ 61 $& $ 95 $ \\
\hline
$ 4 $& $ 133 $& $ 95 $ \\
\hline
$ 5 $& $ 127 $& $ 36 $ \\
\hline
$ 6 $& $ 57 $& $ 77 $ \\
\hline
$ 7 $& $ 167 $& $ 95 $ \\
\hline
$ 8 $& $ 202 $& $ 68 $ \\
\hline
$ 9 $& $ 211 $& $ 118 $ \\
\hline
$ 10 $& $ 31 $& $ 56 $ \\
\hline
$ 11 $& $ 71 $& $ 90 $ \\
\hline
$ 12 $& $ 43 $& $ 95 $ \\
\hline
$ 13 $& $ 197 $& $ 95 $ \\
\hline
$ 14 $& $ 221 $& $ 8 $ \\
\hline
$ 15 $& $ 201 $& $ 92 $ \\
\hline
$ 16 $& $ 179 $& $ 50 $ \\
\hline
$ 17 $& $ 148 $& $ 105 $ \\
\hline
$ 18 $& $ 146 $& $ 95 $ \\
\hline
$ 19 $& $ 65 $& $ 96 $ \\
\hline
$ 20 $& $ 44 $& $ 84 $ \\
\hline
$ 21 $& $ 226 $& $ 8 $ \\
\hline
$ 22 $& $ 87 $& $ 103 $ \\
\hline
$ 23 $& $ 84 $& $ 77 $ \\
\hline
$ 24 $& $ 149 $& $ 36 $ \\
\hline
$ 25 $& $ 210 $& $ 96 $ \\
\hline
$ 26 $& $ 161 $& $ 36 $ \\
\hline
$ 27 $& $ 78 $& $ 62 $ \\
\hline
$ 28 $& $ 79 $& $ 92 $ \\
\hline
$ 29 $& $ 182 $& $ 36 $ \\
\hline
$ 30 $& $ 201 $& $ 77 $ \\
\hline
$ 31 $& $ 162 $& $ 106 $ \\
\hline
$ 32 $& $ 153 $& $ 105 $ \\
\hline
$ 33 $& $ 222 $& $ 105 $ \\
\hline
$ 34 $& $ 159 $& $ 95 $ \\
\hline
$ 35 $& $ 176 $& $ 42 $ \\
\hline
$ 36 $& $ 206 $& $ 42 $ \\
\hline
$ 37 $& $ 155 $& $ 103 $ \\
\hline
$ 38 $& $ 217 $& $ 68 $ \\
\hline
$ 39 $& $ 207 $& $ 103 $ \\
\hline
$ 40 $& $ 226 $& $ 92 $ \\
\hline
$ 41 $& $ 201 $& $ 47 $ \\
\hline
$ 42 $& $ 210 $& $ 84 $ \\
\hline
$ 43 $& $ 212 $& $ 116 $ \\
\hline
$ 44 $& $ 223 $& $ 117 $ \\
\hline
$ 45 $& $ 219 $& $ 92 $ \\
\hline
$ 46 $& $ 174 $& $ 108 $ \\
\hline
$ 47 $& $ 169 $& $ 79 $ \\
\hline
$ 48 $& $ 47 $& $ 103 $ \\
\hline
$ 49 $& $ 42 $& $ 47 $ \\
\hline
$ 50 $& $ 211 $& $ 105 $ \\
\hline
\end{tabular}
\begin{tabular}{|c||c|c|}
\multicolumn{3}{c}{$ 51  \leqslant i\leqslant 100 $}\\
\hline
{\bf $i'$}&{\bf $k$}&{\bf $l$}\\
\hline
\hline
$ 51 $& $ 71 $& $ 56 $ \\
\hline
$ 52 $& $ 212 $& $ 36 $ \\
\hline
$ 53 $& $ 85 $& $ 103 $ \\
\hline
$ 54 $& $ 149 $& $ 105 $ \\
\hline
$ 55 $& $ 205 $& $ 79 $ \\
\hline
$ 56 $& $ 209 $& $ 90 $ \\
\hline
$ 57 $& $ 199 $& $ 79 $ \\
\hline
$ 58 $& $ 208 $& $ 77 $ \\
\hline
$ 59 $& $ 174 $& $ 110 $ \\
\hline
$ 60 $& $ 175 $& $ 43 $ \\
\hline
$ 61 $& $ 196 $& $ 117 $ \\
\hline
$ 62 $& $ 207 $& $ 78 $ \\
\hline
$ 63 $& $ 209 $& $ 105 $ \\
\hline
$ 64 $& $ 217 $& $ 79 $ \\
\hline
$ 65 $& $ 160 $& $ 43 $ \\
\hline
$ 66 $& $ 204 $& $ 68 $ \\
\hline
$ 67 $& $ 187 $& $ 43 $ \\
\hline
$ 68 $& $ 202 $& $ 62 $ \\
\hline
$ 69 $& $ 215 $& $ 118 $ \\
\hline
$ 70 $& $ 227 $& $ 106 $ \\
\hline
$ 71 $& $ 194 $& $ 79 $ \\
\hline
$ 72 $& $ 202 $& $ 90 $ \\
\hline
$ 73 $& $ 218 $& $ 47 $ \\
\hline
$ 74 $& $ 182 $& $ 106 $ \\
\hline
$ 75 $& $ 228 $& $ 118 $ \\
\hline
$ 76 $& $ 175 $& $ 36 $ \\
\hline
$ 77 $& $ 87 $& $ 117 $ \\
\hline
$ 78 $& $ 199 $& $ 95 $ \\
\hline
$ 79 $& $ 178 $& $ 102 $ \\
\hline
$ 80 $& $ 217 $& $ 74 $ \\
\hline
$ 81 $& $ 51 $& $ 90 $ \\
\hline
$ 82 $& $ 176 $& $ 103 $ \\
\hline
$ 83 $& $ 209 $& $ 103 $ \\
\hline
$ 84 $& $ 217 $& $ 62 $ \\
\hline
$ 85 $& $ 223 $& $ 90 $ \\
\hline
$ 86 $& $ 163 $& $ 78 $ \\
\hline
$ 87 $& $ 52 $& $ 36 $ \\
\hline
$ 88 $& $ 201 $& $ 90 $ \\
\hline
$ 89 $& $ 223 $& $ 47 $ \\
\hline
$ 90 $& $ 221 $& $ 92 $ \\
\hline
$ 91 $& $ 194 $& $ 103 $ \\
\hline
$ 92 $& $ 148 $& $ 95 $ \\
\hline
$ 93 $& $ 198 $& $ 74 $ \\
\hline
$ 94 $& $ 226 $& $ 117 $ \\
\hline
$ 95 $& $ 181 $& $ 36 $ \\
\hline
$ 96 $& $ 48 $& $ 90 $ \\
\hline
$ 97 $& $ 32 $& $ 110 $ \\
\hline
$ 98 $& $ 178 $& $ 77 $ \\
\hline
$ 99 $& $ 147 $& $ 103 $ \\
\hline
$ 0 $& $ 50 $& $ 50 $ \\
\hline
\end{tabular}
\begin{tabular}{|c||c|c|}
\multicolumn{3}{c}{$ 101  \leqslant i\leqslant 150 $}\\
\hline
{\bf $i'$}&{\bf $k$}&{\bf $l$}\\
\hline
\hline
$ 1 $& $ 224 $& $ 118 $ \\
\hline
$ 2 $& $ 143 $& $ 105 $ \\
\hline
$ 3 $& $ 182 $& $ 79 $ \\
\hline
$ 4 $& $ 227 $& $ 116 $ \\
\hline
$ 5 $& $ 226 $& $ 74 $ \\
\hline
$ 6 $& $ 224 $& $ 90 $ \\
\hline
$ 7 $& $ 196 $& $ 68 $ \\
\hline
$ 8 $& $ 180 $& $ 43 $ \\
\hline
$ 9 $& $ 90 $& $ 90 $ \\
\hline
$ 10 $& $ 188 $& $ 106 $ \\
\hline
$ 11 $& $ 184 $& $ 47 $ \\
\hline
$ 12 $& $ 206 $& $ 102 $ \\
\hline
$ 13 $& $ 205 $& $ 43 $ \\
\hline
$ 14 $& $ 193 $& $ 110 $ \\
\hline
$ 15 $& $ 211 $& $ 103 $ \\
\hline
$ 16 $& $ 199 $& $ 84 $ \\
\hline
$ 17 $& $ 219 $& $ 68 $ \\
\hline
$ 18 $& $ 226 $& $ 36 $ \\
\hline
$ 19 $& $ 205 $& $ 84 $ \\
\hline
$ 20 $& $ 67 $& $ 90 $ \\
\hline
$ 21 $& $ 230 $& $ 8 $ \\
\hline
$ 22 $& $ 178 $& $ 110 $ \\
\hline
$ 23 $& $ 232 $& $ 118 $ \\
\hline
$ 24 $& $ 229 $& $ 77 $ \\
\hline
$ 25 $& $ 66 $& $ 92 $ \\
\hline
$ 26 $& $ 173 $& $ 110 $ \\
\hline
$ 27 $& $ 93 $& $ 42 $ \\
\hline
$ 28 $& $ 210 $& $ 95 $ \\
\hline
$ 29 $& $ 69 $& $ 90 $ \\
\hline
$ 30 $& $ 84 $& $ 62 $ \\
\hline
$ 31 $& $ 200 $& $ 90 $ \\
\hline
$ 32 $& $ 199 $& $ 36 $ \\
\hline
$ 33 $& $ 211 $& $ 79 $ \\
\hline
$ 34 $& $ 67 $& $ 74 $ \\
\hline
$ 35 $& $ 225 $& $ 84 $ \\
\hline
$ 36 $& $ 216 $& $ 90 $ \\
\hline
$ 37 $& $ 77 $& $ 108 $ \\
\hline
$ 38 $& $ 205 $& $ 77 $ \\
\hline
$ 39 $& $ 186 $& $ 84 $ \\
\hline
$ 40 $& $ 173 $& $ 43 $ \\
\hline
$ 41 $& $ 50 $& $ 90 $ \\
\hline
$ 42 $& $ 49 $& $ 36 $ \\
\hline
$ 43 $& $ 213 $& $ 68 $ \\
\hline
$ 44 $& $ 200 $& $ 78 $ \\
\hline
$ 45 $& $ 63 $& $ 47 $ \\
\hline
$ 46 $& $ 95 $& $ 62 $ \\
\hline
$ 47 $& $ 146 $& $ 103 $ \\
\hline
$ 48 $& $ 201 $& $ 103 $ \\
\hline
$ 49 $& $ 83 $& $ 108 $ \\
\hline
$ 50 $& $ 209 $& $ 36 $ \\
\hline
\end{tabular}%

}
\label{swc:8_5}
\end{table}


\begin{thebibliography}{1}

\bibitem{Magma}
W. Bosma, J. J. Cannon, C. Fieker, A. Steel (eds.). 
\newblock Handbook of Magma functions.
\newblock {Edition 2.18}, 5017 pages. 2013.

\bibitem{Brouwer:SpectraGraphs}
A.E. Brouwer and W.H. Haemers.
\newblock {\em Spectra of Graphs}.
\newblock Universitexts. Springer, 2012.

\bibitem{BusseNeum}
F.C. Bussemaker and A. Neumaier.
\newblock Exceptional graphs with smallest eigenvalue $-2$ and related problems.
\newblock {\em Math. Comp.}, 59(200):583--608, 1992.

\bibitem{CGSS}
P.J. Cameron, J.M. Goethals, E.E. Shult and J.J. Seidel.
\newblock Line graphs, root systems, and elliptic geometry.
\newblock {\em J. Algebra} 43(1):305--327, 1976.

\bibitem{CRS:LMS314}
D. Cvetkovi\'c, P. Rowlinson and S. Simi\'c.
\newblock {\em Spectral Generalizations of Line Graphs}.
\newblock London Math. Soc. Lecture Note Ser. 314. Cambridge Univ. Press, 2004.

\bibitem{Doob:LineGraphsEigen73}
M. Doob.
\newblock An interrelation between line graphs, eigenvalues, and matroids.
\newblock {\em J. Combin. Theory B}, 15:40--50, 1973.

\bibitem{DoobCvetkovic:LAA79}
M. Doob and D. Cvetkovi\'c.
\newblock On spectral characterizations and embeddings of graphs.
\newblock {\em Linear Algebra Appl.}, 27:17--26, 1979.

\bibitem{Gleitz:KNSconj}
A.-S. Gleitz.
\newblock On the KNS conjecture in type $E$.
\newblock {\em Preprint.}, arXiv:1307.2738v1.

\bibitem{Godsil:AlgGraph}
C.D. Godsil and G. Royle.
\newblock {\em Algebraic Graph Theory}.
\newblock Graduate Texts in Mathematics. New York: Springer, 2000.

\bibitem{Hoffman:LimitLeastEig1977}
A.J. Hoffman.
\newblock On limit points of the least eigenvalue of a graph.
\newblock {\em Ars Combin.}, 3:3--14, 1977.

\bibitem{Hoffman:lt-2}
A.J. Hoffman.
\newblock On graphs whose least eigenvalue exceeds $-1-\sqrt{2}$.
\newblock {\em Linear Algebra Appl.}, 16:153--165, 1977.

\bibitem{HornJohnson}
R.A. Horn and C.R. Johnson.
\newblock {\em Matrix Analysis}.
\newblock Cambridge University Press. 1985

\bibitem{Ishihara}
T. Ishihara. 
\newblock Signed graphs associated with the lattice $A_n$. 
\newblock {\em J. Math. Univ. Tokushima}, Vol. 36 (2002), 1--6. 

\bibitem{-3}
H.J. Jang, J. Koolen, A. Munemasa and T. Taniguchi.
\newblock On fat Hoffman graphs with smallest eigenvalue at least $-3$.
\newblock {\em Ars Math. Contemp.}, 7:105--121, 2014.

\bibitem{McKee:IntSymCyc07}
J. McKee and C. Smyth.
\newblock Integer symmetric matrices having all their eigenvalues in the
  interval $[-2,2]$.
\newblock {\em J. Algebra}, 317:260--290, 2007.


\bibitem{-1-tau}
A. Munemasa, Y. Sano and T. Taniguchi.
\newblock Fat Hoffman graphs with smallest eigenvalue at least $-1-\tau$,
\newblock {\em Ars Math. Contemp.}, 7:247--262, 2014. 

\bibitem{gt-3}
A. Munemasa, Y. Sano and T. Taniguchi.
\newblock Fat Hoffman graphs with smallest eigenvalue greater than $-3$.
\newblock {\em Discrete Appl. Math.}, to appear. 

\bibitem{EJC1987}
G.R. Vijayakumar.
\newblock Signed graphs represented by $D_\infty$,
\newblock {\em European J. Combin.}, 8:103--112, 1987.

\bibitem{ArsComb1993}
G.R. Vijayakumar.
\newblock Algebraic equivalence of signed graphs with all eigenvalues
$\geq-2$.
\newblock {\em Ars Combin.}, 35:173--191, 1993.

\bibitem{Woo:ltHoff}
R. Woo and A. Neumaier.
\newblock On graphs whose smallest eigenvalue is at least $-1-\sqrt{2}$.
\newblock {\em Linear Algebra Appl.}, 226-228:577--591, 1995.

\end{thebibliography}
\end{document}